\documentclass[12pt]{amsart}
\usepackage{amssymb}
\usepackage{mathtools}
\usepackage{txfonts}
\usepackage{eucal}
\usepackage{extarrows}
\usepackage{color}

\usepackage{color}

\usepackage[all]{xy}

\usepackage{tikz}

\usepackage{graphicx}
\usepackage{float}

\usepackage{xspace}

\usepackage[a4paper,body={16.3cm,22.8cm},centering]{geometry}

\ifpdf
\usepackage[colorlinks,final,backref=page,hyperindex]{hyperref}
\else
\usepackage[colorlinks,final,backref=page,hyperindex,hypertex]{hyperref}
\fi
\usepackage[active]{srcltx} %SRC Specials for DVI Searching

\newcommand{\nc}{\newcommand}
\newcommand{\delete}[1]{}

\nc{\mlabel}[1]{\label{#1}}  % Use this to suppress names
\nc{\mcite}[1]{\cite{#1}}  % Use this to suppress names
\nc{\mref}[1]{\ref{#1}}  % Use this to suppress names
\nc{\meqref}[1]{\eqref{#1}}  % Use this to suppress names
\nc{\mbibitem}[1]{\bibitem{#1}} % Use this to show number name
\delete{
\nc{\mlabel}[1]{\label{#1}  % Use the next two lines to show names
{\hfill \hspace{1cm}{\small\tt{{\ }\hfill(#1)}}}}
\nc{\mcite}[1]{\cite{#1}{\small{\tt{{\ }(#1)}}}}  % Use this lines to show names
\nc{\mref}[1]{\ref{#1}{{\tt{{\ }(#1)}}}}  % Use this lines to show names
\nc{\meqref}[1]{\eqref{#1}{{\tt{{\ }(#1)}}}}  % Use this lines to show names
\nc{\mbibitem}[1]{\bibitem[\bf #1]{#1}} % Use this to show name
}

\newtheorem{theorem}{Theorem}[section]
\newtheorem{prop}[theorem]{Proposition}
\newtheorem{lemma}[theorem]{Lemma}
\newtheorem{coro}[theorem]{Corollary}

\theoremstyle{definition}
\newtheorem{defn}[theorem]{Definition}
\newtheorem{remark}[theorem]{Remark}
\newtheorem{exam}[theorem]{Example}
\newtheorem{prop-def}{Proposition-Definition}[section]

\newcommand\alphlist{a,b,c,d,e,f,g,h,i,j,k,l,m,n,o,p,q,r,s,t,u,v,w,x,y,z}
\newcommand\Alphlist{A,B,C,D,E,F,G,H,I,J,K,L,M,N,O,P,Q,R,S,T,U,V,W,X,Y,Z}
\newcommand\getcmds[3]{\expandafter\newcommand\csname #2#1\endcsname{#3{#1}}}
\makeatletter
\@for\x:=\alphlist\do{\expandafter\getcmds\expandafter{\x}{frak}{\mathfrak}}% \fraka,\frakb,...
\@for\x:=\Alphlist\do{\expandafter\getcmds\expandafter{\x}{frak}{\mathfrak}}% \frakA,\frakB,...
\makeatother

\nc{\name}[1]{{\bf #1}}

\nc{\bfk}{{\bf k}}
\font\cyr=wncyr10

\newfont{\scyr}{wncyr10 scaled 550}
\nc{\sha}{\mbox{\cyr X}}
\nc{\ssha}{\mbox{\bf \scyr X}}

\nc{\Id}{\mathrm{Id}}
\nc{\lbar}[1]{\overline{#1}}

\nc{\leaf}{\mathrm{leaf}}

\usetikzlibrary{calc}

\nc{\fax}{\mathcal{F}(X)} %angularly decorated forests
\nc{\pfa}{\mathcal{PF}}
\nc{\shpr}{\diamond}    %Shuffle product
\nc{\ot}{\otimes}      % tensor product
\nc{\dep}{\mathrm{dep}} %depth of tree
\nc{\var}{\varepsilon} %varepsilon
\nc{\id}{\mathrm{id}}  %id
\nc{\set}{\mathbf{Set}} % the cat of set
\nc{\vect}{\mathbf{Vect}} % the cat of vector space
\nc{\spep}{\mathbf{p}} %species p
\nc{\speq}{\mathbf{q}} %species q
\nc{\spei}{\mathbf{1}_{\bfk}} %species1
\nc{\Hom}{\mathrm{Hom}}
\nc{\calc}{\mathcal{C}}
\nc{\adf}{\mathrm{ADF}}
\nc{\pure}{simple\xspace}
\nc{\Coinv}{\mathrm{Coinv}}
\nc{\ff}{\mathcal{F}}
\nc{\bff}{\widehat{\mathcal{F}}}
\nc{\com}{\mathrm{Com}}
\nc{\comp}{\mathrm{Comp}}
\nc{\rbs}{\text{Rota-Baxter species}}

\newcommand{\un}{\mathbf{1}}
\newcommand{\F}{\widehat{\mathcal{F}}}
\newcommand{\N}{\mathbb{N}}
\renewcommand{\geq}{\geqslant}
\renewcommand{\leq}{\leqslant}

\nc{\li}[1]{\textcolor{red}{#1}}
\nc{\lir}[1]{\textcolor{red}{Li:#1}}
\nc{\peng}[1]{\textcolor{purple}{Peng:#1}}
\nc{\yi}[1]{\textcolor{cyan}{Yi:#1}}
\nc{\revise}[1]{\textcolor{red}{#1}}

\begin{document}

\title[Species of Rota-Baxter algebras by rooted trees and twisted bialgebras]{Species of Rota-Baxter algebras by rooted trees, twisted bialgebras and Fock functors}
	
%=========================================================================

\author{Lo\"\i c Foissy}
\address{Univ. Littoral C\^ote d'Opale, UR 2597 LMPA, Laboratoire de Math\'ematiques Pures et Appliqu\'ees Joseph Liouville F-62100 Calais, France}
\email{loic.foissy@univ-littoral.fr}

\author{Li Guo}
\address{
	Department of Mathematics and Computer Science,
	Rutgers University,
	Newark, NJ 07102, United States}
\email{liguo@rutgers.edu}

\author{Xiao-song Peng}
\address{School of Mathematics and Statistics,
Jiangsu Normal University, Xuzhou, Jiangsu 221116, P.\,R. China}
\email{pengxiaosong3@163.com}

\author{Yunzhou Xie}
\address{Department of Mathematics, Imperial College London, London SW7 2AZ, UK}
\email{yx3021@ic.ac.uk}

\author{Yi Zhang}
\address{School of Mathematics and Statistics,
	Nanjing University of Information Science \& Technology, Nanjing, Jiangsu 210044, P.\,R. China}
\email{zhangy2016@nuist.edu.cn}

\date{\today}
\begin{abstract} As a fundamental and ubiquitous combinatorial notion, species has attracted sustained interest, generalizing from set-theoretical combinatorial to algebraic combinatorial and beyond. The Rota-Baxter algebra is one of the algebraic structures with broad applications from Renormalization of quantum field theory to integrable systems and multiple zeta values. Its interpretation in terms of monoidal categories has also recently appeared. This paper studies species of Rota-Baxter algebras, making use of the combinatorial construction of free  Rota-Baxter algebras in terms of angularly decorated trees and forests. The notion of \pure angularly decorated forests is introduced for this purpose and the resulting Rota-Baxter species is shown to be free. Furthermore, a twisted bialgebra structure, as the bialgebra for species, is established on this free Rota-Baxter species. Finally, through the Fock functor, another proof of the bialgebra structure on free Rota-Baxter algebras is obtained.
\end{abstract}

\subjclass[2010]{
18M80, % species, monoidal categories
17B38, % RBO and YBE
18D10, % Category theory; homological algebra
05C05, % trees
16S10, % Rings determined by universal properties (free algebras, coproducts, adjunction of inverses, etc.)
16T10, % bialgebras
08B20	% free algebras
\!}

\keywords{Species, monoidal category, Rota-Baxter algebra, rooted tree, bialgebra, Fock functor}

\maketitle

\tableofcontents

\setcounter{section}{0}

\allowdisplaybreaks
%========================================================================

\section{Introduction}
This paper introduces the species of Rota-Baxter algebras, by equipping the species of angularly decorated rooted trees with the operations for a Rota-Baxter algebra, which is further enriched to a twisted bialgebra structure. Through the Fock functors, we obtain another proof of the bialgebra structure on free Rota-Baxter algebras on rooted trees.

\subsection{Species and monoidal categories}
The concept of species was introduced by A. Joyal in his seminal paper~\mcite{Jo}, formulated as functors from the category of finite sets with bijective maps to the category of sets, providing a way to categorize combinatorial objects based on their structures. The theory was further developed as a fundamental tool in combinatorics and several other areas of mathematics, providing deep insights into the nature of combinatorial objects and their relationships~\mcite{ALL, BLL, FGHW, Ges, MM}.

Vector species can be viewed as objects in a category, with morphisms given by natural transformations. This category is endowed with a symmetric monoidal structure, which gives rise to various interesting notions, such as monoids, comonoids, bimonoids, and Hopf monoids within the context of species. A general structured theory of these algebraic structures, internal to the functor category of species, has been developed by Aguiar and Mahajan~\mcite{AM10,AM12}.
Furthermore, there are natural functors that transform vector species into graded vector spaces, and Hopf monoids in species into graded Hopf algebras. The constructions and properties of Hopf algebras, when lifted to the species level, are often more accessible and easier to generalize through the use of these functors. In the other direction, through Fock functors, these enriched species with additional structures can be regarded as lifts of these structures in their original (graded) form.

As further enrichment and extensions of the classical species theory, species with additional structures, such as the associative algebra and bialgebra, were studied recently~\mcite{Car, Foi, Mar, Mar2, Nor, Ta, Wh}.
In particular, ~\mcite{Foi} studied bialgebras in the category of linear species under the name of twisted bialgebras and gave a
species version of the chromatic symmetric series and chromatic polynomials.
Later, Tamaroff~\mcite{Ta} gave a Koszul duality theory between twisted algebras and coalgebras on the twisted coalgebra that gives rise to their cohomology theory to give a new alternative description of it. At the same time,  Carlier~\mcite{Car} showed that Schmitt's hereditary species induce monoidal decomposition spaces and exemplify Schmitt's bialgebra construction as a case of the general bialgebra construction on these spaces.
Furthermore, Cohen-Macaulay Hopf monoids in the category of species was studied in~\mcite{Wh}.

Species theory has found applications in various areas, including graph theory, algebra, and combinatorial enumeration, providing deep insights into the nature of combinatorial objects and their relationships.
Recently, Norledge~\mcite{Nor} developed an algebraic framework for perturbative quantum field theory (pQFT) based on Joyal's combinatorial species. He demonstrated that fundamental pQFT structures can be viewed as algebraic structures within species, constructed using the Cauchy monoidal product.

\subsection{Species and Rota-Baxter algebras}

Rota-Baxter algebra was introduced by G. Baxter from fluctuation theory in probability~\mcite{Ba}.  Afterwards,  it attracted attentions of prominent mathematicians such as Atkinson, Cartier and Rota~\mcite{Ca,Ro,Ro2}.  In the early 1980s,
Rota-Baxter operators on Lie algebras were independently discovered by Semenov-Tian-Shansky in~\mcite{STS} as the operator form of the classical Yang-Baxter equation~\mcite{BD}, named after the physicists C. N. Yang and R. J. Baxter.

Over the past two decades, the study of Rota-Baxter algebras has undergone an extraordinary renascence thanks to broad applications and connections such as  Renormalization of quantum field theory~\mcite{CK}, Yang-Baxter equations~\mcite{Bai,BG,Bor,GS,GG}, operads~\mcite{Ag,BBGN,WZ},
combinatorics~\mcite{AMo,YGT}, Lie groups~\mcite{GLS,LS}, deformation theories~\mcite{Da,TBGS} and Hopf algebras~\mcite{AGKO,BGM,ZGG}. See~\mcite{Gub} for an introduction.

The category of species is a monoidal category~\mcite{AM10}. Thus monoids and Hopf monoids can be defined for species.  On the other hand, Rota-Baxter operators have been defined for monoidal categories in~\mcite{Ag2,St}. 

In this paper, we combine species and Rota-Baxter algebras to define Rota-Baxter species by the Cauchy product, yielding a species with a Rota-Baxter algebra structure. In particular, we equipped angularly decorated forests with a free Rota-Baxter species, making use of the combinatorial constructions of free Rota-Baxter algebras given by rooted trees and angularly decorated forests~\mcite{AMo,EG,Gub,ZGG,ZXG}. A twisted bialgebra on the free Rota-Baxter species is also obtained, which recovers the biaglebra structure on free Rota-Baxter algebras. In a related work~\mcite{GGRZZ}, the rings of various species (set species, linear species and localized set species) are equipped with the structures of an integro-differential ring~\mcite{RR} or a modified integro-differential ring~\mcite{GGL}. 

Here is an outline of the paper. In Section~\mref{sec:a}, we recall the needed preliminaries on linear species. Then  we use the Rota-Baxter operators on monoidal categories to define Rota-Baxter species and give the notion of free Rota-Baxter species. We construct a Rota-Baxter species on what we call \pure angularly decorated forests and show that it is the free Rota-Baxter species.

In Section~\mref{sec:ad}, we show that a species of bialgebras, namely a twisted bialgebra, can be given on the above-obtained free Rota-Baxter species. This is accomplished by making use of the universal property of the free Rota-Baxter species, which allows us to obtain the existence of the coproduct and its coassociativity.

In Section~\mref{sec:fock}, we show that the (Bosonic) Fock functor from a species to its associated graded space can be enrich to a functor from a Rota-Baxter species to a Rota-Baxter algebra on the associated grades space. As a special case, the Fock functor applied to the free Rota-Baxter species recovers the construction of free Rota-Baxter algebras previously obtained by a direct construction~\mcite{EG,Gub,ZXG}.

\smallskip

\noindent
{\bf Notations.} In this paper, we fix a base field $\bfk$ of characteristic zero. All linear spaces, linear maps and tensor products are taken on this base field.

\section{Rota-Baxter algebras on species}\mlabel{sec:a}
This section introduces Rota-Baxter species and constructs free Rota-Baxter species by angularly decorated trees.

\subsection{Species and Rota-Baxter species}
First we recall basic notions and facts about species that will be used later. For more details, see~\cite{AM12,BLL,Jo}.

\begin{defn}
Let $\set$ be the category whose objects are finite sets and whose
morphisms are bijections, and let $\vect$ be the category whose
objects are vector spaces and whose morphisms are linear maps.
\begin{enumerate}
\item A {\bf species} is a functor
$$\spep:\set \longrightarrow \vect.$$

\item For a species $\spep$, its value on a finite set $X$ is denoted by $\spep[X]$.  Its value on a bijection  $\sigma: X \rightarrow Y$ is given by
    $$\spep[\sigma]:\spep[X] \rightarrow \spep[Y].$$
\item For any finite set $X$, we have $\spep[\id_X]=\id_{\spep[X]}$. For any bijections $\sigma: X \rightarrow Y$ and $\tau: Y \rightarrow Z$, we have
    $$\spep[\tau \circ \sigma]=\spep[\tau] \circ \spep[\sigma].$$

\item A {\bf morphism between species} is a natural transformation of the functors for the two species. More precisely, for two species $\spep$ and $\speq$, a {\bf morphism $f: \spep \rightarrow \speq$ of species} is a collection of linear maps
\begin{align*}
f_X: \spep[X] \rightarrow \speq[X],
\end{align*}
one for each finite set $X$ such that, for each bijection $\sigma: X \rightarrow Y$ of finite sets, the following diagram commutes.
\[
\begin{tikzpicture}
\node (P0) at (150:2cm) {$\spep[X]$};
\node (P1) at (210:2cm) {$\spep[Y]$} ;
\node (P2) at (330:2cm) {$\speq[Y]$};
\node (P3) at (30:2cm) {$\speq[X]$};
\draw
(P0) edge[->] node[left] {$\spep[\sigma]$} (P1)
(P1) edge[->] node[below] {$f_Y$} (P2)
(P0) edge[->] node[above] {$f_X$} (P3)
(P3) edge[->] node[right] {$\speq[\sigma]$} (P2);
\end{tikzpicture}
\]

\item If the linear maps $f_X: \spep[X] \rightarrow \speq[X]$ in a morphism between species are bijections, then the morphism is called an {\bf isomorphism between species}. 

\item For any element $k \in \bfk$ and any morphisms of species $f,g: \spep \rightarrow \speq$, define the morphisms $k f, f+g: \spep \rightarrow \speq$ by
\begin{align*}
(kf)_X :=k f_X \,\, \text{ and} \,\, (f+g)_X :=f_X+ g_X,
\end{align*}
for each finite set $X$.
Then the category of species is a \bfk-linear category.
\item Given two species $\spep$ and $\speq$, the {\bf Cauchy product} $\spep \ot \speq$ is defined by
\begin{align*}
\spep \ot \speq [X] := \bigoplus_{X_1 \sqcup X_2=X} \spep[X_1] \ot \speq[X_2],
\end{align*}
for any finite set $X$.

\item Denote by $\spei$ the species defined by
\begin{equation}
\spei[X] :=
\left\{
\begin{array}{ll}
\bfk, & \text{if $X=\emptyset$}, \\
0 , &  \text{otherwise}.\\
\end{array}
\right .
\end{equation}
Then for any species $\spep$, there is
$$\spep \ot \spei= \spep =\spei \ot \spep.$$
Moreover, the category of species is a monoidal category under the Cauchy product, with the unit object $\spei$~\cite{AM12}.
\end{enumerate}
\end{defn}

\begin{exam} Here are two examples of species~\mcite{Foi}.
\begin{enumerate}
\item The species $\com$ is defined by $\com[X]=\bfk$ for any finite set $X$ and $\com[\sigma]=\id_{\bfk}$ for any bijection $\sigma: X \rightarrow Y$.

\item For any finite set $X$, a set composition of $X$ is a finite sequence $(X_1, \ldots,X_k)$ of nonempty subset of $X$ such that $X_1 \sqcup \cdots \sqcup X_k=X$. Denote by $\comp[X]$ the set of set compositions of $X$. Then define the species $Comp$ as follows. For any finite set $X$, define $Comp[X]$ to be the space generated by $\comp[X]$; while for any set bijection $\sigma: X \rightarrow Y$ and $(X_1, \ldots, X_k) \in \comp[X]$, define
$$Comp[\sigma](X_1, \ldots, X_k)=(\sigma(X_1), \ldots, \sigma(X_k)) \in Comp[Y].$$
\end{enumerate}
\end{exam}

As the category of species is a monoidal category, the notion of algebras is defined in the category of species as follows~\mcite{Foi}.

\begin{defn}~\mcite{Foi}
An {\bf algebra in the category of species}, or {\bf twisted algebra}, is a triple $(\spep, m,\ell)$, where $\spep$ is a species, $m: \spep \ot \spep \rightarrow \spep$ and $\ell: \spei \rightarrow \spep$ are morphisms of species such that
\begin{align*}
m \circ (m \ot \id_{\spep})=m \circ (\id_{\spep} \ot m) \, \text{ and } \, m \circ (\ell \ot \id_{\spep})=\id_{\spep}= m \circ (\id_{\spep} \ot \ell).
\end{align*}
\end{defn}

Equivalently, a twisted algebra $\spep$ satisfies that for any two finite sets $X$ and $Y$, there is a map $m_{X,Y}: \spep[X] \ot \spep[Y] \rightarrow \spep[X \sqcup Y]$ such that
\begin{enumerate}
\item For any two bijections $\sigma: X \rightarrow X'$ and $\tau: Y \rightarrow Y'$ of finite sets,
\begin{align*}
m_{X',Y'} \circ (\spep[\sigma] \ot \spep[\tau])= \spep[\sigma \sqcup \tau] \circ m_{X,Y}: \spep[X] \ot \spep[Y] \rightarrow \spep[X' \sqcup Y'],
\end{align*}
where $\sigma \sqcup \tau: X \sqcup Y \rightarrow X' \sqcup Y'$ is the bijection induced by $\sigma$ and $\tau$.

\item For finite sets $X,Y,Z$,
\begin{align*}
m_{X \sqcup Y, Z} \circ (m_{X,Y} \ot \id_{\spep[Z]})=m_{X, Y \sqcup Z} \circ (\id_{\spep[X]} \ot m_{Y,Z}).
\end{align*}

\item There is an element $1_{\spep} \in \spep[\emptyset]$ such that for any finite set $X$ and any element $u \in \spep[X]$,
\begin{align*}
m_{\emptyset, X} (1_{\spep} \ot u)=x =m_{X,\emptyset}(u \ot 1_{\spep}).
\end{align*}
\end{enumerate}

\begin{remark}~\mcite{Foi} \label{rem:definem}
Let $\spep$ be a twisted algebra and $X$ be a finite set. For $(X_1, \ldots, X_k) \in Comp[X]$, there is a map $m_{X_1, \ldots, X_k}$ from $\spep[X_1] \otimes \cdots \otimes \spep[X_k]$ to $\spep[X]$ inductively defined by
\begin{align*}
m_X=&\ \id_{\spep[X]},\\
m_{X_1, \ldots, X_k}=&\ m_{X_1 \sqcup X_2, X_3, \ldots, X_k} \circ (m_{X_1, X_2} \otimes \id_{\spep[X_3]} \otimes \cdots \otimes \id_{\spep[X_k]}).
\end{align*}
By associativity, if $(X_1, \ldots, X_{k+\ell}) \in Comp[X]$, then
\begin{align*}
m_{X_1 \sqcup \cdots \sqcup X_k, X_{k+1} \sqcup \cdots \sqcup X_{k+\ell}} \circ (m_{X_1, \ldots, X_{k}} \otimes m_{X_{k+1}, \ldots, X_{k+ \ell}})=m_{X_1, \ldots, X_{k+\ell}}.
\end{align*}
\end{remark}

\begin{exam}
The species $\spei$ is a twisted algebra, where for $k_1 \in \spei[\emptyset]=\bfk$ and $k_2 \in \spei[\emptyset]=\bfk$,
\begin{align*}
m_{\emptyset,\emptyset}(k_1 \ot k_2)=k_1 k_2 \in \spei[\emptyset]=\bfk.
\end{align*}
The unit of $\spei$ is $1_{\bfk} \in \bfk=\spei[\emptyset] $.
\end{exam}

Rota-Baxter operators have been defined for monoidal categories in~\mcite{Ag2,St}. As the category of species is a monoidal category, we give the concept of Rota-Baxter species as follows.

\begin{defn}
Fix a scalar $\lambda$. A {\bf Rota-Baxter algebra of weight $\lambda$ in the category of species}, or {\bf Rota-Baxter species of weight $\lambda$}, is a quadruple $(\spep, m, \ell, R)$, where $(\spep, m, \ell)$ is a twisted algebra and $R: \spep \rightarrow \spep$ is a morphism of species such that the following diagram commutes:
\[
\begin{tikzpicture}
\node (P0) at (20:2.4cm) {$\spep \ot \spep $};
\node (P1) at (170:5cm) {$\spep \ot \spep$} ;
\node (P2) at (190:5cm) {$\spep \ot \spep$};
\node (P3) at (350:5cm) {$\spep$};
\node (P4) at (10:5cm) {$\spep$};
\draw
(P1) edge[->] node[above] {$R \ot \id_{\spep}+\id_{\spep}\ot R+\lambda \id_{\spep} \ot \id_{\spep}$} (P0)
(P1) edge[->] node[left] {$R \ot R$} (P2)
(P2) edge[->] node[below] {$m$} (P3)
(P0) edge[->] node[above] {$m$} (P4)
(P4) edge[->] node[right] {$R$} (P3);
\end{tikzpicture}
\]
\end{defn}
In more detail, the commutative diagram means that, for any two finite sets $X,Y$ and any elements $x \in \spep[X], y \in \spep[Y]$,
\begin{align*}
m_{X,Y} \circ (R_X \ot R_Y)(x \ot y)=R_{X \sqcup Y} \circ m_{X,Y} \circ (R_X \ot \id_{\spep[Y]}+ \id_{\spep[X]} \ot R_Y + \lambda \id_{\spep[X]} \ot \id_{\spep[Y]})(x \ot y).
\end{align*}

\begin{defn}
Let $(\spep, m, \ell, R)$ and $(\spep', m', \ell', R')$ be two $\rbs$  of the same weight, and let $f: \spep \to \spep'$ be a morphisms of species. We shall say that $f$ is a \textbf{morphism of \rbs} if 
\begin{align*}
f \circ m&=m' \circ (f \ot f),&\ell'&=f \circ \ell,& f \circ R&=R' \circ f.
\end{align*}
\end{defn}

Inspired by the concept of free Rota-Baxter algebras~\mcite{Gub},  we  also give the concept of free Rota-Baxter species as follows.

\begin{defn}
Let $P$ be a species. Let $(F,m,\ell, R)$ be a $\rbs$ and $i: P \rightarrow F$ be a morphism of species. We say that $F$ with the morphism $i: P \rightarrow F$ is a {\bf free \rbs} on $P$ if for any $\rbs$ $(Q,m,1_Q,R')$ and any morphism $g: P \rightarrow Q$ of species, there is a unique morphism of $\rbs$  $\overline{g}: F \rightarrow Q$ such that $g= \overline{g} \circ i$.
\end{defn}

\begin{remark}
If it exists, the pair $(F,i)$ is unique up to an isomorphism. Indeed, if $(F,i)$ and $(G,j)$ are two free Rota-Baxter species on $P$, using the species morphism $j:P\rightarrow G$, we obtain a morphism of Rota-Baxter species $\overline{j}:F\longrightarrow G$
such that $\overline{j}\circ i=j$. Using the species morphism $i:P\rightarrow F$, we obtain a morphism of Rota-Baxter species $\overline{i}:G\longrightarrow F$
such that $\overline{i}\circ j=i$. Then $\overline{i}\circ \overline{j}$ is an endomorphism of Rota-Baxter species of $F$ such that $\overline{i}\circ \overline{j}\circ i=i$. By unicity of such an endomorphism, $\overline{i}\circ \overline{j}=\id_F$.
Similarly, $\overline{j}\circ \overline{i}=\id_G$, so $\overline{i}$ and $\overline{j}$ are two isomorphisms, inverse one from the other.  
\end{remark}

\subsection{Species of \pure angularly decorated forests}
In this subsection, we first recall the concept of angularly decorated forests~\mcite{EG,Gub,ZXG}. Then as an example of species, we give the species of \pure angularly decorated forests.

\begin{defn}
Let $X$ be a set.
\begin{enumerate}
\item  An {\bf angularly $X$-decorated rooted tree} is a planar rooted tree together with a decoration of the angles between adjacent leaves by elements of $X$.

\item For a rooted forest $F=T_1\cdots T_n$, as the concatenation of rooted trees $T_1,\ldots,T_n$, an {\bf angle of the forest $F$} is defined to be either the angle between two leaves of a rooted tree, or the space between two adjacent rooted trees $T_i$ and $T_{i+1}, 1\leq i\leq n-1$.

\item An {\bf angularly $X$-decorated rooted forest} is a planar rooted forest $F$ together with a decoration of the angles of $F$ by elements of $X$. Let $\fax$ denote the set of angularly $X$-decorated rooted forests.
\end{enumerate}
\end{defn}

The one leaf tree $\bullet$ is trivially an angularly $X$-decorated rooted tree with no angle to decorate.
Here are some examples of angularly $X$-decorated trees
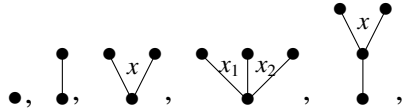
\begin{align*}
\begin{tikzpicture}[scale = 0.6] %one
\node at (0,0.1) {$\bullet$};
\node at (0.3,0) {,};
\end{tikzpicture}
\begin{tikzpicture}[scale = 0.6] %two
\node at (0,0) {$\bullet$};
\node at (0,1) {$\bullet$};
\draw (0,0)-- (0,1);
\node at (0.3,0) {,};
\end{tikzpicture}
\begin{tikzpicture}[scale = 0.6] %three
\node at (0,1) {$\bullet$};
\node at (0.5,0) {$\bullet$};
\node at (1,1) {$\bullet$};
\draw (0,1)--(0.5,0)--(1,1);
\node at (0.5,0.7) [scale=0.8]{$x$};
\node at (1.3,0) {,};
\end{tikzpicture}
\begin{tikzpicture}[scale = 0.6] %four
\node at (0,1) {$\bullet$};
\node at (1,1) {$\bullet$};
\node at (1,0) {$\bullet$};
\node at (2,1) {$\bullet$};
\draw (0,1)--(1,0)--(1,1);
\draw (1,0)--(2,1);
\node at (0.6,0.7) [scale=0.8]{$x_1$};
\node at (1.4,0.7) [scale=0.8]{$x_2$};
\node at (2.3,0) {,};
\end{tikzpicture}
\begin{tikzpicture}[scale = 0.6] %five
\node at (0.5,2) {$\bullet$};
\node at (1,1) {$\bullet$};
\node at (1,0) {$\bullet$};
\node at (1.5,2) {$\bullet$};
\draw (0.5,2)--(1,1)--(1,0);
\draw (1,1)--(1.5,2);
\node at (1,1.7) [scale=0.8]{$x$};
\node at (1.8,0) {,};
\end{tikzpicture}
\end{align*}
and here are some examples of angularly $X$-decorated forests
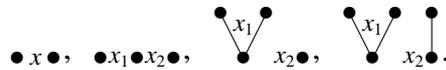
\begin{align*}
\begin{tikzpicture}[scale = 0.6] % one
\node at (0,0) {$\bullet$};
\node at (0.8,0) {$\bullet$};
\node at (0.4,0) [scale=0.8]{$x$};
\node at (1.1,0) {,};
\end{tikzpicture}
\begin{tikzpicture}[scale=0.6] % two
\node at (0,0) {$\bullet$};
\node at (0.8,0) {$\bullet$};
\node at (1.6,0) {$\bullet$};
\node at (0.4,0) [scale=0.8] {$x_1$};
\node at (1.2,0) [scale=0.8] {$x_2$};
\node at (1.9,0) {,};
\end{tikzpicture}
\begin{tikzpicture}[scale=0.6] % three
\node at (0,1) {$\bullet$};
\node at (0.5,0) {$\bullet$};
\node at (1,1) {$\bullet$};
\draw (0,1)--(0.5,0)--(1,1);
\node at (0.5,0.7) [scale=0.8]{$x_1$};
\node at (1.4,0) [scale=0.8] {$x_2$};
\node at (1.8,0) {$\bullet$};
\node at (2.1,0) {,};
\end{tikzpicture}
\begin{tikzpicture}[scale=0.6] % four
\node at (0,1) {$\bullet$};
\node at (0.5,0) {$\bullet$};
\node at (1,1) {$\bullet$};
\draw (0,1)--(0.5,0)--(1,1);
\node at (0.5,0.7) [scale=0.8]{$x_1$};
\node at (1.4,0) [scale=0.8] {$x_2$};
\node at (1.8,0) {$\bullet$};
\node at (1.8,1) {$\bullet$};
\draw (1.8,0)-- (1.8,1);
\node at (2.1,0) {.};
\end{tikzpicture}
\end{align*}
In general, an angularly $X$-decorated rooted forest is of the form
\begin{align*}
T_1 x_1 T_2 x_2 \cdots x_{n-1} T_n,
\end{align*}
where $x_1, \ldots, x_{n-1} \in X$ and $T_1, \ldots, T_n$ are angularly $X$-decorated rooted trees.

\begin{remark}
Note that the number of angles of a rooted forest $F$ is precisely $\ell-1$ for the number $\ell$ of leaves of $F$.
\end{remark}

With this remark in mind, the following description of angularly $X$-decorated forests from~\mcite{Gub} will be convenient to apply. For an angularly $X$-decorated forest $F$, let $u(F)$ be the underlying forest of $F$ and let $d(F)=(x_1,\cdots,x_{\ell-1})\in X^{\ell-1}$ denote the angular decorations of $F$. Then $F$ can be written in the form
\begin{equation}
	F=(u(F),d(F)).
\mlabel{eq:fpres}
\end{equation}
Conversely, every pair $(S,\vec{x})$ consisting of a forest $S$ with $\ell$ leaves and a vector $\vec{x}\in X^{\ell-1}$ uniquely determines an angular $X$-decorated rooted forest $F=F_{(S,\vec{x})}$, by decorating the angles of $S$ with the entries of $\vec{x}$ in the order from the left to the right.

Let $X$ be a finite set. An angularly $X$-decorated rooted forest $F \in \fax$ is called {\bf \pure} if the angles of $F$ are decorated by all the distinct elements of $X$. Since the number of angles of $F$ is $\ell-1$ if the number of leaves $F$ is $\ell$, this means that $|X|=\ell-1$ and each element of $X$ appears exactly once in the angles of $F$. Denote by $\pfa (X)$ the set of all \pure angularly $X$-decorated rooted forests. Thus for a fixed ordering $\vec{x}$ of the elements of $X$, we have
$$\pfa= \{F=(u(F),\tau(\vec{x}))\,|\, \tau\in S_{\ell-1}\}$$
for the permutation group $S_{\ell-1}$ on $\ell-1$ letters.

Define a functor $\adf: \set \rightarrow \vect$ by taking
$$
\adf[X] :=\bfk \pfa (X)
$$
for any finite set $X$, and for any bijection $\sigma: X \rightarrow X'$ of finite sets, defining
$$\adf[\sigma] : \adf[X] \rightarrow \adf[Y],\quad F \mapsto \adf[\sigma](F),
$$
where $F$ is a \pure angularly $X$-decorated rooted forest and $\adf[\sigma](F)$ is the \pure angularly $X'$-decorated rooted forest obtained by replacing the decoration $x \in X$ for any given angle of $F$ by the decoration $\sigma(x)$ for the same angle.

In terms of the description of angular decorated forests in Eq.~\meqref{eq:fpres}, let $F=(u(F),d(F))$ with the underlying rooted forest $u(F)$ of $F$ and the angular decoration vector $\vec{x}=(x_1,\ldots,x_{\ell-1})\in X^{\ell -1}$. Then we simply have
\begin{equation}
	\adf[\sigma](F)=(u(F),\sigma(\vec{x})),
\mlabel{eq:rbspec}
\end{equation}
where $\sigma(\vec{x})=(\sigma(x_1),\dots,\sigma(x_{\ell-1}))$.

By the definition of $\adf$, we have
\begin{align*}
\adf[\id_X]&=\id_{\adf[X]}, &\adf[\sigma \circ \tau]&=\adf[\sigma] \circ \adf[\tau],
\end{align*}
for the identity set map $\id_X: X \rightarrow X$, and bijective maps $\sigma: Y \rightarrow Z, \tau: X \rightarrow Y$. Hence $\adf: \set \rightarrow \vect$ is a species.

\subsection{Free $\rbs$ on the species of \pure angularly decorated forests} \label{sec:ac}
In this subsection, we construct a $\rbs$ on the species of simple angularly decorated forests and show that this $\rbs$ is a free $\rbs$. First we recall the Rota-Baxter algebra structure on angularly decorated forests~\mcite{EG,Gub}.

For an angularly decorated rooted tree $T$, define the {\bf depth $\dep(T)$} of $T$ to be the maximal length of the paths from the root to the leaves of the tree. For an angularly decorated forest $F = T_1 x_1 T_2 \cdots x_n T_n$ with rooted trees $T_1, \ldots, T_n$, define the {\bf depth $\dep(F)$} of $F$ to be the maximum of the depths of the trees $T_1, \ldots, T_k$.

There is a {\bf grafting operator} $B^+:\fax \rightarrow \fax$ which sends an angularly $X$-decorated rooted forest $T_1 x_1 T_2 x_2 \cdots x_{n-1} T_n$ to a new angularly $X$-decorated rooted tree that is obtained by adding a new root, and then adding an edge from the new root to the root of each of the trees $T_1, \ldots, T_n$. For example,
\begin{align*}
B^+\left(
\begin{tikzpicture}[scale = 0.6]
\node at (0,0) {$\bullet$};
\node at (0.8,0) {$\bullet$};
\node at (0.4,0) [scale=0.8]{$x$};
\end{tikzpicture}\right)=
\begin{tikzpicture}[scale = 0.6]
\node at (0,1) {$\bullet$};
\node at (0.5,0) {$\bullet$};
\node at (1,1) {$\bullet$};
\draw (0,1)--(0.5,0)--(1,1);
\node at (0.5,0.7) [scale=0.8]{$x$};
\end{tikzpicture} \,\, \text{ and } \,\,
B^{+}\left(\begin{tikzpicture}[scale=0.6]
\node at (0,0) {$\bullet$};
\node at (0.8,0) {$\bullet$};
\node at (1.6,0) {$\bullet$};
\node at (0.4,0) [scale=0.8] {$x_1$};
\node at (1.2,0) [scale=0.8] {$x_2$};
\end{tikzpicture}\right)=
\begin{tikzpicture}[scale = 0.6]
\node at (0,1) {$\bullet$};
\node at (1,1) {$\bullet$};
\node at (1,0) {$\bullet$};
\node at (2,1) {$\bullet$};
\draw (0,1)--(1,0)--(1,1);
\draw (1,0)--(2,1);
\node at (0.6,0.7) [scale=0.8]{$x_1$};
\node at (1.4,0.7) [scale=0.8]{$x_2$};
\end{tikzpicture}.
\end{align*}
In terms of the representation $F=(u(F),d(F))$ in Eq.~\meqref{eq:fpres}, we have
$B^+(F)=(B^+(u(F)),d(F))$.

Let $\bfk \fax$ be the free \bfk-module generated by $\fax$. To equip $\bfk \fax$ with a Rota-Baxter algebra structure, a multiplication $\shpr$ is defined on $\bfk \fax$ as follows~\mcite{EG,Gub}. For $F,F' \in \bfk \fax$, $F \shpr F'$ is defined by induction on $\dep=\dep(F)+\dep(F')$. If $\dep=0$, then $F,F'$ are of the form
\begin{align*}
F=\bullet x_1 \bullet \cdots x_m \bullet \,\, \text{ and } \,\, F'= \bullet y_1 \bullet \cdots y_n \bullet,
\end{align*}
with the convention that $F=\bullet$ if $m=0$ and $F'=\bullet$ if $n=0$. Then define
\begin{align*}
F \shpr F' := \bullet x_1 \bullet \cdots \bullet x_m \bullet y_1 \bullet \cdots  \bullet y_n \bullet.
\end{align*}
For the induction step of $\dep>0$, let
\begin{equation}
F=T_1 x_1 T_2 \cdots x_m T_{m+1} \,\, \text{ and } \,\, F'=T'_1 y_1 T'_2 \cdots y_n T'_{n+1}.
\mlabel{eq:fdec}
\end{equation}
First consider the case of $F=T_1,F'=T'_1$ being angularly decorated trees. Then define
\begin{equation}
F \shpr F' :=
\left\{
\begin{array}{ll}
\bullet, & \text{if $T_1=T'_1=\bullet$}, \\
T_1 , &  \text{if $T'_1=\bullet$},\\
T'_1, &  \text{if $T_1=\bullet$},\\
B^+(T_1 \shpr \overline{T'}_1)+B^+(\overline{T}_1 \shpr T'_1)+ \lambda B^+(\overline{T}_1 \shpr \overline{T'}_1), & \text{if $T_1=B^+(\overline{T}_1)$ and $T'_1=B^+(\overline{T'}_1)$}.
\end{array}
\right. \mlabel{eq:pro}
\end{equation}
Here in the last case, each term on the right hand side is well defined by the induction hypothesis on the sum of depths.
In the general case of $F$ and $F'$ as in Eq.~\meqref{eq:fdec}, define
\begin{align*}
F \shpr F' :=T_1 x_1 T_2 \cdots x_m (T_{m+1} \shpr T'_1)y_1 T'_2 \cdots y_n T'_{n+1},
\end{align*}
where the product in the middle is given by Eq.~\meqref{eq:pro} and the rest by concatenation. Equivalently, one can define a multiplication $\shpr_u$ on the underlying forests by the same procedure, then for $F=(u(F),d(F))$ and $F'=(u(F'),d(F'))$, define
$$F\shpr F'=\Big(u(F)\shpr_u u(F'), \big(d(F),d(F')\big)\Big).$$ Here $\big(d(F),d(F')\big)$ denotes the vector concatenation of $d(F)$ and $d(F')$.

Finally expand the operations $\shpr$ and $B^+$ to $\bfk \fax$ by (bi-)linearity. Then the following result is obtained.

\begin{theorem}\cite{EG,Gub}
Let $X$ be a set. Let $i_{x}:X \to \bfk\fax,\text{ }x\mapsto \bullet x \bullet$ be the set map. The quadruple $(\bfk \fax,\shpr,B^+,i_x)$ is the free non-commutative unitary Rota-Baxter algebra generated by $X$.
\end{theorem}

To construct the twisted algebra on the species $\adf$, we first prove the following lemma.
\begin{lemma}
Let $X,Y$ be two finite sets. Then for  $F \in \adf[X]$ and $G \in \adf[Y]$, we have
\begin{align*}
F \diamond G \in \adf[X \sqcup Y],
\end{align*}
where the product $\diamond$ is the one defined in Eq.~\meqref{eq:pro}.
\label{lem:close}
\end{lemma}

\begin{proof}
Let
\begin{equation}
F=T_1 x_1 T_2 \cdots x_m T_{m+1} \in \adf[X], \quad G=T'_1 y_1 T'_2 \cdots y_n T'_{n+1} \in \adf[Y]
\mlabel{eq:fgdec}
\end{equation}
for some $m, n\geq 0$ as in Eq.~\meqref{eq:fdec}. As in the definition of $\diamond$ in Eq.~\meqref{eq:pro}, we prove the result by induction on $\dep(F)+\dep(G)$. If $\dep(F)+ \dep(G)=0$, then
\begin{align*}
F=\bullet x_1 \bullet \cdots x_m \bullet, \, G= \bullet y_1 \bullet \cdots y_n \bullet,
\end{align*}
and so the result holds trivially. For the induction step of $\dep(F)+ \dep(G)>0$, we prove the result by first considering $F=T_1, G=T'_1$ for \pure angularly $X$-decorated tree $T_1$ and \pure angularly $Y$-decorated tree $T_1'$. If $T_1=\bullet$ or $T'_1=\bullet$, then the result holds directly. In the remaining case, we may write $T_1=B^+(\overline{T}_1)$ and $T'_1=B^+(\overline{T'}_1)$ for some $\overline{T}_1\in ADF[X]$, $\overline{T'}_1 \in ADF[Y]$, then
\begin{align*}
&\ F \diamond G= B^+(\overline{T}_1) \diamond B^+(\overline{T'}_1)=B^+(T_1 \shpr \overline{T'}_1)+B^+(\overline{T}_1 \shpr T'_1)+ \lambda B^+(\overline{T}_1 \shpr \overline{T'}_1).
\end{align*}
By the induction hypothesis on the depth,
$T_1 \shpr \overline{T'}_1,\overline{T}_1 \shpr T'_1$ and $\overline{T}_1 \shpr \overline{T'}_1$ are in $\adf[X \sqcup Y]$.
Then  $F \diamond G$ is in  $\adf[X \sqcup Y]$. For general $F$ and $F'$ in Eq.~\meqref{eq:fgdec}, we have $T_{m+1} \in \adf[X']$ and $T'_1 \in \adf[Y']$, for some $X' \subseteq X$ and $Y' \subseteq Y$.  By the definition of $\diamond$,
\begin{align*}
F \diamond G= T_1 x_1T_2 \cdots x_m (T_{m+1} \diamond T'_1) y_1 T'_2 \cdots y_n T'_{n+1}.
\end{align*}
Since $T_{m+1} \diamond T'_1 \in \adf[X' \sqcup Y']$ by the previous case and other decorations are intact, it follows that $F \diamond G$ is in $\adf[X \sqcup Y]$.
\end{proof}

For finite sets $X, Y$, and elements $F \in \adf[X], G \in \adf[Y]$, define
\begin{align}
m_{X,Y}(F,G) :=F \diamond G.
\label{eq:mult}
\end{align}
Then we have the following proposition.

\begin{prop}
The species $\adf$ is a twisted algebra when it is equipped with the multiplication $m_{X,Y}$ defined by Eq.~\meqref{eq:mult} and the unit $\bullet \in \adf[\emptyset]$.
\label{prop:talgebra}
\end{prop}

\begin{proof}
By the definition of $\adf$, we observe  that for any two bijections $\sigma: X \rightarrow X'$ and $\tau: Y \rightarrow Y'$ of finite sets, there is
\begin{align*}
m_{X',Y'} \circ (\adf[\sigma] \ot \adf [\tau])= \adf[\sigma \sqcup \tau] \circ m_{X,Y}.
\end{align*}
Also, for finite sets $X,Y,Z$ and $F \in \adf[X], G \in \adf[Y], H \in \adf[Z]$, we have
\begin{align*}
&\ m_{X \sqcup Y, Z} \circ (m_{X,Y} \ot \id_{\adf[Z]})(F \ot G \ot H)\\
=& (F \diamond G) \diamond H\\
=& F \diamond (G \diamond H)\\
=&\ m_{X, Y \sqcup Z} \circ (\id_{\adf[X]} \ot m_{Y,Z})(F \ot G \ot H).
\end{align*}
Finally, for a finite set $X$ and $F \in \adf[X]$,
\begin{align*}
m_{\emptyset, X} (\bullet \ot F)= \bullet \diamond F=F= F \diamond \bullet =m_{X,\emptyset}(F \ot \bullet).
\end{align*}
Thus the required conditions are verified.
\end{proof}

Now for any finite set $X$, define a map
\begin{equation}
R_X: \adf[X] \rightarrow \adf[X], F \mapsto B^+(F),
\mlabel{eq:rbo}
\end{equation}
where $B^+$ is the grafting operator recalled above. Then for any bijection $\sigma : X \rightarrow X'$ and $F \in \adf[X]$, there is
\begin{align*}
\adf[\sigma] \circ R_X (F)=\adf[\sigma](B^+(F))=B^+(\adf[\sigma](F))=R_{X'} \circ \adf[\sigma](F).
\end{align*}
Hence $R: \adf \rightarrow \adf$ is a morphism of species. Moreover, we have

\begin{prop}
The twisted algebra $(\adf, m, \bullet)$, with the morphism $R$ of species, is a  Rota-Baxter species.
\mlabel{prop:trb}
\end{prop}

\begin{proof}
For any finite sets $X,Y$ and any elements $F \in \adf[X], G \in \adf[Y]$, we have
\begin{align*}
&\ m_{X,Y} \circ (R_X \circ R_Y)(F \ot G)\\
=&m_{X,Y}(B^+(F) \ot B^+(G))\\
=&B^+(F) \diamond B^+(G)\\
=&\ B^+(B^+(F) \diamond G+ F \diamond B^+(G)+ \lambda F \diamond G) \,\, \text{(by the definition of $\diamond$)}\\
=&\ R_{X \sqcup Y} \circ m_{X,Y} \circ (R_X \ot \id_{\adf[Y]}+ \id_{\adf[X]} \ot R_Y + \lambda \id_{\adf[X]} \ot \id_{\adf[Y]})(F \ot G),
\end{align*}
as required.
\end{proof}

We define a species $O$ by
\begin{equation}
O[X] :=
\left\{
\begin{array}{ll}
\bfk \{ \bullet x \bullet \}, & \text{if $X=\{ x\}$ is a singleton set}, \\
0 , &  \text{otherwise}.\\
\end{array}
\right .
\end{equation}
If $\sigma : X=\{ x \} \rightarrow Y=\{ y\}$ is a bijection, then 
\[
O[\sigma]:\left\{\begin{array}{rcl}
O[X] &\longrightarrow& O[Y]\\
\bullet x \bullet &\longmapsto& \bullet y \bullet.
\end{array}\right. \]
There is a morphism $i_O: O \rightarrow \adf$ of species, where $(i_O)_X: O[X] \rightarrow \adf[X]$ is the natural inclusion if $X$ is a singleton set and otherwise $i_O=0$.

Now we show that the $\rbs$ $\adf$ with the morphism $i:=i_O: O \rightarrow \adf$ is a free $\rbs$ on $O$.

\begin{theorem}
The $\rbs$ $\adf$ together with the morphism $i:O\to \adf$ of species is the free $\rbs$ on $O$.
\mlabel{thm:freerb}
\end{theorem}

\begin{proof}
For the proof, we prove that for any $\rbs$ $(Q,m,1_Q,R')$ and any morphism $g: O \rightarrow Q$ of species, there is a unique morphism of $\rbs$ $\overline{g}: \adf \rightarrow Q$ such that $g= \overline{g} \circ i$.

For a finite set $X$, define $\overline{g}_X: \adf[X] \rightarrow Q[X]$ by induction on $|X|$. If $|X|=0$, then $X=\emptyset$. For a rooted forest $F \in \adf[\emptyset]$, we define $\overline{g}_{\emptyset}(F)$ by induction on $\dep(F)$. If $\dep(F)=0$, then $F=\bullet$ and define
\begin{align*}
\overline{g}_{\emptyset}(\bullet) := 1_{Q}.
\end{align*}
For the induction step of $\dep(F)>0$, note that $F \in \adf[\emptyset]$ and so $F=B^+(\overline{F})$ for some $\overline{F} \in \adf[\emptyset]$. Then define
\begin{align*}
\overline{g}_{\emptyset} (F) :=R'_{\emptyset} (\overline{g}(\overline{F})).
\end{align*}
Hence $\overline{g}_{\emptyset}(F)$ are defined for all $F \in \adf[\emptyset]$. For the induction step of $|X|>0$, take a rooted forest $F \in \adf[X]$ and define $\overline{g}_X(F)$ by a second induction on $\dep(F)$. If $\dep(F)=0$, then $F$ is of the form $\bullet x_1 \bullet \cdots \bullet x_m \bullet$, where $X=\{ x_1, \ldots, x_m\}$. If $m=1$ and $F=\bullet x_1 \bullet$, then define
\begin{align*}
\overline{g}_X(F):=g_{\{ x_1\}}(\bullet x_1 \bullet).
\end{align*}
If $m>1$, then define
\begin{align*}
\overline{g}_X(F):=m_{\{x_1 \}, \{x_2 \}, \ldots, \{ x_m\}}(g_{\{ x_1\} }(\bullet x_1 \bullet) \ot \cdots \ot  g_{\{ x_m\} }(\bullet x_m \bullet)),
\end{align*}
where $m_{\{x_1 \}, \{x_2 \}, \ldots, \{ x_m\}}$ is the map from $Q[\{ x_1\}] \ot \cdots \ot Q[\{ x_m\}]$ into $Q[X]$ inductively defined as in Remark~\mref{rem:definem}.

For the induction step of $\dep(F)>0$, let $F=T_1 x_1 T_2 \cdots T_m x_m T_{m+1} \in \adf[X]$,  we define $\overline{g}_X(F)$ by induction on $m$. If $m=0$, then $F=B^+(\overline{F})$ for some $\overline{F} \in \adf[X]$. Then define
\begin{align*}
\overline{g}_X (F)=R'_X( \overline{g}_X(\overline{F})),
\end{align*}
where $\overline{g}_X(\overline{F})$ is defined by induction hypothesis.

For the induction step of $m>0$,
assume $T_1 \in \adf[X_1], \ldots, T_m \in \adf[X_m], T_{m+1} \in \adf[X_{m+1}]$ and $X_1 \sqcup \cdots \sqcup X_{m+1} \sqcup \{ x_1, \ldots, x_{m+1}\}=X$. Then define
\begin{align*}
\overline{g}_X(F)=m_{X_1, \{ x_1\}, \ldots, X_m, \{x_m\}, X_{m+1}} ( \overline{g}_{X_1}(T_1) \ot \overline{g}_{\{ x_1\} }(\bullet x_1 \bullet) \ot \cdots \ot \overline{g}_{\{  x_m\}}(\bullet x_m \bullet) \ot \overline{g}_{X_{m+1}}(T_{m+1})),
\end{align*}
where $m_{X_1, \{ x_1\}, \ldots, X_m, \{x_m\}, X_{m+1}}$ is the map from $Q[X_1] \ot Q[\{ x_1\}] \ot \cdots \ot Q[\{ x_m\}] \ot Q[X_{m+1}]$ into $Q[X]$ inductively defined as in Remark~\mref{rem:definem}.

By the construction of $\overline{g}$, it is the unique morphism of twisted Rota-Baxter algebras such that $g=\overline{g} \circ i$.
\end{proof}

\section{Twisted bialgebras on the species of \pure angularly decorated forests}\mlabel{sec:ad}
In this section, applying the universal property of the free Rota-Baxter species, a bialgebra structure is given to the free $\rbs$ on \pure angularly decorated forests.

\subsection{Rota-Baxter structures on tensor powers}

We first construct an augmentation on $\adf$ which will be the counit in the bialgebra construction.

The species $1_{\bfk}$ is a Rota-Baxter species, with the product of $\bfk$ and the null Rota-Baxter operator. Thus the universal property of $\adf$ gives a unique morphism of Rota-Baxter species $$\epsilon:\adf \longrightarrow 1_\bfk$$ 
such that
for any singleton set $\{x\}$, $\epsilon_{\{x\}}(\bullet x \bullet)=0$. 

\begin{lemma}
The species morphism $\epsilon$ thus defined is an augmentation: for any finite set $X$ and $F\in \adf[X]$, we have 
\begin{equation}
\epsilon_X (F) :=
\left\{
\begin{array}{ll}
1_{\bfk} , & \text{if $X=\emptyset$ and $F=\bullet$}, \\
0 , &  \text{otherwise}.\\
\end{array}
\right . 
\end{equation}
\mlabel{lem:eam}
\end{lemma}

\begin{proof}
As $\epsilon$ is a morphism of Rota-Baxter species, $\epsilon_\emptyset(\bullet)=1_\bfk$. Its kernel is an ideal of the twisted algebra $\adf$, containing the elements $\bullet x \bullet$, and the image of $R$ as $\epsilon\circ R=0$. 
An easy induction on the depth shows that this kernel contains all the \pure angularly decorated  forests  except $\bullet$. 
\end{proof}

For later applications, we give some properties of the morphism $\epsilon$ as follows.

\begin{lemma}
For any finite sets $X, Y$ and any $F \in \adf[X], G \in \adf [Y]$, we have
\begin{enumerate}
\item \mlabel{it:a} $\epsilon_{X \sqcup Y}(R_X(F) \diamond G)=\epsilon_{X \sqcup Y}(F \diamond R_Y(G))=0$;
\item \mlabel{it:b} $R_{X \sqcup Y} (\epsilon_X(F) \bullet \diamond G)=\epsilon_X(F) \bullet \diamond R_Y(G)$, $R_{X \sqcup Y} (F \diamond \epsilon_Y(G) \bullet)=R_X(F) \diamond \epsilon_Y(G) \bullet$;
\item \mlabel{it:c} $\epsilon_{X \sqcup Y}( \epsilon_X(F) \bullet \diamond G)= \epsilon_{X \sqcup Y} (F \diamond \epsilon_Y(G)\bullet)=\epsilon_X(F)\epsilon_Y(G)$.
\end{enumerate}
\mlabel{prop:eam}
\end{lemma}

\begin{proof}
\mref{it:a} By the definition of $\diamond$, we see that the expansions of $R_X(F) \diamond G$ and $F \diamond R_Y(G)$ do not include the term $\bullet$. Hence
$\epsilon_{X \sqcup Y}(R_X(F) \diamond G)=\epsilon_{X \sqcup Y}(F \diamond R_Y(G))=0$.

\mref{it:b} If $X =\emptyset$ and $F=\bullet$, then
$$R_{X \sqcup Y}(\epsilon_X(F) \bullet \diamond G)= R_Y(\bullet \diamond G)=R_Y(G)=\bullet \diamond R_Y(G)=\epsilon_X(F)\bullet \diamond R_Y(G).$$
If not, then
$$R_{X \sqcup Y} (\epsilon_X(F) \bullet \diamond G)=R_{X \sqcup Y} (0 \diamond G)=0=\epsilon_X(F) \bullet \diamond R_Y(G).$$
Similarly, we have
$$R_{X \sqcup Y} (F \diamond \epsilon_Y(G) \bullet)=R_X(F) \diamond \epsilon_Y(G) \bullet.$$

\mref{it:c} If $X =\emptyset$ and $F=\bullet$, then
$$\epsilon_{X \sqcup Y}( \epsilon_X(F) \bullet \diamond G)= \epsilon_Y(\bullet \diamond G) =\epsilon_Y(G)=\epsilon_X(F) \epsilon_Y(G).$$
If not, then
$$\epsilon_{X \sqcup Y} (\epsilon_X(F) \bullet \diamond G)= \epsilon_Y(0 \diamond G)=0=\epsilon_X(F)\epsilon_Y(G).$$
Similarly, we have
$$\hspace{1cm} \epsilon_{X \sqcup Y} (F \diamond \epsilon_Y(G)\bullet)=\epsilon_X(F)\epsilon_Y(G). \hspace{1cm} \qedhere $$
\end{proof}

Note that~\cite{Foi} if $(\spep, m_{\spep})$ and $(\speq, m_{\speq})$ are two twisted algebras, then $\spep \ot \speq$ is also a twisted algebra with the product
\begin{align*}
m_{\spep \ot \speq} :=(m_{\spep} \ot m_{\speq}) \circ (\id_{\spep} \ot c_{\spep, \speq} \ot \id_{\speq}),
\end{align*}
where $c_{\spep, \speq}: \spep \ot \speq \rightarrow \speq \ot \spep$ is the isomorphism of species
\begin{equation*}
(c_{\spep, \speq})_X:
\left\{
\begin{array}{ll}
\spep \ot \speq[X] & \rightarrow \speq \ot \spep [X], \\
x \ot y  & \mapsto y \ot x  .\\
\end{array}
\right .
\end{equation*}

Since $\adf$ is a twisted algebra, $\adf \ot \adf$ is also a twisted algebra. We further have:
\begin{prop}
The twisted algebra $\adf \ot \adf$ is a $\rbs$, where the Rota-Baxter operator $R^{(2)}: \adf \ot \adf \rightarrow \adf \ot \adf$ is defined by
\begin{equation*}
R^{(2)}_X:
\left\{
\begin{array}{rl}
\adf \ot \adf[X] & \rightarrow \adf \ot \adf[X], \\
 \adf[X_1] \ot \adf[X \setminus X_1]\ni F \ot G   & \mapsto R_{X_1}(F)  \ot G+ \epsilon_{X_1}(F) \bullet \ot R_{X \setminus X_1} (G).\\
\end{array}
\right .
\end{equation*}
\label{prop:ttensor}
\end{prop}

\begin{proof}
If $X_1=\emptyset$ and for $F \ot G \in \adf[X_1] \ot \adf[X \setminus X_1]$ with $F=\bullet$,
\begin{align*}
R^{(2)}_X(F \ot G)=R_{\emptyset}(F) \ot G+\epsilon_{\emptyset}(F) \bullet \ot R_X(G) =B^+(F) \ot G+ \bullet \ot B^+(G) \in \adf[\emptyset] \ot \adf[X].
\end{align*}
Otherwise, $\epsilon_{X_1}(F)=0$ and so
\begin{align*}
R^{(2)}_X(F \ot G)= R_{X_1}(F)  \ot G+ \epsilon_{X_1}(F) \bullet \ot R_{X \setminus X_1} (G)=B^+(F)  \ot G \in \adf [X_1]\ot \adf[X \setminus X_1].
\end{align*}
For any bijection $\sigma: X \rightarrow X'$ and $F \ot G \in \adf \ot \adf[X]$, by the definition of $R^{(2)}$, we find that
\begin{align*}
(\adf \ot \adf[\sigma]) \circ R^{(2)}_X(F \ot G)=R^{(2)}_{X'} \circ ( \adf \ot \adf[\sigma])(F \ot G).
\end{align*}
Hence $R^{(2)}: \adf \ot \adf \to \adf \ot \adf$ is a morphism of species. Now for $F_1 \ot G_1 \in  \adf[X_1] \ot \adf [X \setminus X_1] \subset \adf \ot \adf[X]$ and $F_2 \ot G_2 \in \adf[Y_1] \ot \adf[Y \setminus Y_1] \subset \adf \ot \adf[Y]$, there is
\begin{align*}
&\ R^{(2)}_{X \sqcup Y} \circ (m_{\adf \ot \adf})_{X,Y} \circ (R^{(2)}_X \ot \id+ \id \ot R^{(2)}_Y + \lambda \id \ot \id )((F_1 \ot G_1) \ot (F_2 \ot G_2))\\
%first=
=&\ R^{(2)}_{X \sqcup Y} \circ (m_{\adf \ot \adf})_{X,Y} (R^{(2)}_X(F_1 \ot G_1) \ot (F_2 \ot G_2) + (F_1 \ot G_1) \ot R^{(2)}_Y(F_2 \ot G_2) \\
&\ + \lambda (F_1 \ot G_1) \ot (F_2 \ot G_2))\\
%second=
=&\ R^{(2)}_{X \sqcup Y} \circ (m_{\adf \ot \adf})_{X,Y} \big((R_{X_1}(F_1) \ot G_1 + \epsilon_{X_1}(F_1) \bullet \ot R_{X \setminus X_1}(G_1)) \ot (F_2 \ot G_2) \\
&\ p+(F_1 \ot G_1) \ot (R_{Y_1}(F_2) \ot G_2 + \epsilon_{Y_1}(F_2) \bullet \ot R_{Y \setminus Y_1}(G_2)) + \lambda (F_1 \ot G_1) \ot (F_2 \ot G_2) \big) \\
%third=
=&\ R^{(2)}_{X \sqcup Y} \Big( (R_{X_1}(F_1) \diamond F_2) \ot (G_1 \diamond G_2) +(\epsilon_{X_1}(F_1) \bullet \diamond F_2) \ot (R_{X \setminus X_1}(G_1) \diamond G_2) \\
&\ +(F_1 \diamond R_{Y_1}(F_2)) \ot (G_1 \diamond G_2)+((F_1 \diamond \epsilon_{Y_1}(F_2) \bullet) \ot (G_1 \diamond R_{Y \setminus Y_1}(G_2))+ \lambda (F_1 \diamond F_2) \ot (G_1 \diamond G_2)\Big)\\
%fourth=
=&\ R_{X_1 \sqcup Y_1}(R_{X_1}(F_1) \diamond F_2) \ot (G_1 \diamond G_2)+ \epsilon_{X_1 \sqcup Y_1}(R_{X_1}(F_1) \diamond F_2) \bullet \ot R_{(X \sqcup Y) \setminus (X_1 \sqcup Y_1)}(G_1 \diamond G_2)\\
&\ +R_{X_1 \sqcup Y_1}(\epsilon_{X_1}(F_1) \bullet \diamond F_2) \ot (R_{X \setminus X_1}(G_1) \diamond G_2)\\
&\ + \epsilon_{X_1 \sqcup Y_1} ((\epsilon_{X_1}(F_1) \bullet) \diamond F_2)\bullet \ot R_{ (X \sqcup Y ) \setminus (X_1 \sqcup Y_1) }(R_{X \setminus X_1}(G_1) \diamond G_2)\\
&\ +R_{X_1 \sqcup Y_1}(F_1 \diamond R_{Y_1}(F_2)) \ot (G_1 \diamond G_2)+ \epsilon_{X_1 \sqcup Y_1}(F_1 \diamond R_{Y_1}(F_2))\bullet \ot R_{(X \sqcup Y) \setminus (X_1 \sqcup Y_1) }(G_1 \diamond G_2)\\
&\ +R_{X_1 \sqcup Y_1}(F_1 \diamond (\epsilon_{Y_1}(F_2) \bullet)) \ot (G_1 \diamond R_{Y \setminus Y_1}(G_2))\\
&\ +\epsilon_{X_1 \sqcup Y_1}(F_1 \diamond (\epsilon_{Y_1}(F_2) \bullet)) \bullet \ot R_{ (X \sqcup Y) \setminus (X_1 \sqcup Y_1) }(G_1 \diamond R_{Y \setminus Y_1}(G_2))\\
&\ +\lambda R_{X_1 \sqcup Y_1}(F_1 \diamond F_2) \ot (G_1 \diamond G_2)+ \lambda \epsilon_{X_1 \sqcup Y_1}(F_1 \diamond F_2)\bullet \ot R_{ (X \sqcup Y) \setminus (X_1 \sqcup Y_1) }(G_1 \diamond G_2)\\
%fifth=
=&\ \Big(R_{X_1 \sqcup Y_1}(R_{X_1}(F_1) \diamond F_2)+R_{X_1 \sqcup Y_1}(F_1 \diamond R_{Y_1}(F_2)) +\lambda R_{X_1 \sqcup Y_1}(F_1 \diamond F_2)  \Big)\ot (G_1 \diamond G_2)\\
&\ +(\epsilon_{X_1}(F_1) \bullet \diamond R_{Y_1}(F_2)) \ot (R_{X \setminus X_1}(G_1) \diamond G_2)+(R_{X_1}(F_1) \diamond \epsilon_{Y_1}(F_2) \bullet) \ot (G_1 \diamond R_{Y \setminus Y_1}(G_2))\\
&\ +(\epsilon_{X_1}(F_1) \epsilon_{Y_1}(F_2))\bullet \ot \Big(R_{\{X \sqcup Y \} \setminus \{X_1 \sqcup Y_1 \}}(R_{X \setminus X_1}(G_1) \diamond G_2)+R_{\{X \sqcup Y \} \setminus \{X_1 \sqcup Y_1 \}}(G_1 \diamond R_{Y \setminus Y_1}(G_2))\\
&\ +\lambda R_{\{X \sqcup Y \} \setminus \{X_1 \sqcup Y_1 \}}(G_1 \diamond G_2)  \Big) \quad \text{(by Lemmas~\mref{lem:eam} and ~\mref{prop:eam})}\\
=&\ (R_{X_1}(F_1) \diamond R_{Y_1}(F_2)) \ot (G_1 \diamond G_2)+(\epsilon_{X_1}(F_1) \bullet \diamond R_{Y_1}(F_2)) \ot (R_{X \setminus X_1}(G_1) \diamond G_2)\\
&\ +(R_{X_1}(F_1) \diamond \epsilon_{Y_1}(F_2) \bullet) \ot (G_1 \diamond R_{Y \setminus Y_1}(G_2))+(\epsilon_{X_1}(F_1) \epsilon_{Y_1}(F_2))\bullet \ot (R_{X \setminus X_1}(G_1) \diamond R_{Y \setminus Y_1}(G_2))\\
=&\ (m_{\adf \ot \adf})_{X,Y}\Big( \left(R_{X_1}(F_1) \ot G_1+ \epsilon_{X_1}(F_1) \bullet \ot R_{X \setminus X_1}(G_1) \right) \ot \left(R_{Y_1}(F_2) \ot G_2+ \epsilon_{Y_1}(F_2) \bullet \ot R_{Y \setminus Y_2}(G_2) \right) \Big)\\
=&\ (m_{\adf \ot \adf})_{X,Y}\circ(R^{(2)}_{X} \ot R^{(2)}_{Y})((F_1 \ot G_1) \ot (F_2 \ot G_2)),
\end{align*}
and hence $R^{(2)}$ satisfies the Rota-Baxter identity.
\end{proof}

Based on $\adf$ being a $\rbs$ with $R$, and $\adf \ot \adf$ being a $\rbs$ with $R^{(2)}$, we define
$R^{(3)}: \adf \ot \adf \ot \adf \rightarrow \adf \ot \adf \ot \adf$ by
\begin{align*}
&\ R^{(3)}_{X_1,X_2,X \setminus(X_1 \sqcup X_2)} (F \ot G \ot H):=R_{X_1}(F)  \ot G \ot H+ \epsilon_{X_1}(F) \bullet \ot R^{(2)}_{X \setminus X_1} (G \ot H)\\
=&\ R_{X_1}(F)  \ot G \ot H+\epsilon_{X_1}(F) \bullet \ot R_{X_2}(G) \ot H+ \epsilon_{X_1}(F) \bullet \ot \epsilon_{X_2}(G) \bullet \ot R_{X \setminus(X_1 \sqcup X_2)}(H),
\end{align*}
for $F \ot G \ot H \in \adf[X_1] \ot \adf[X_2] \ot \adf[X \setminus(X_1 \sqcup X_2)]$. Then by a similar argument for the proof of Proposition~\mref{prop:ttensor}, we obtain

\begin{prop}
The species $\adf \ot \adf \ot \adf$ equipped with the operator $R^{(3)}$ is a Rota-Baxter species.
\mlabel{prop:trb2}
\end{prop}

\subsection{The twisted bialgebra structure}
We now apply the universal property of the free Rota-Baxter species $\adf$ and the Rota-Baxter structure on the tensor powers of this algebra to obtain a bialgebra structure on $\adf$.

\begin{defn}~\mcite{Foi}
A {\bf coalgebra in the category of species}, or {\bf twisted coalgebra}, is a triple $(\spep, \Delta,\varepsilon)$, where $\spep$ is a species, $\Delta: \spep \rightarrow \spep \ot \spep$ and $\varepsilon: \spep \rightarrow \spei$ are morphisms of species such that
\begin{align*}
(\Delta \ot \id_{\spep}) \circ \Delta= (\id_{\spep} \ot \Delta) \circ \Delta \, \text{ and } \,   (\varepsilon \ot \id_{\spep}) \circ \Delta=\id_{\spep}=  (\id_{\spep} \ot \varepsilon) \circ \Delta.
\end{align*}
\end{defn}
Equivalently, a twisted coalgebra $\spep$ satisfies that for any two finite sets $X$ and $Y$, there is a map $\Delta_{X,Y}: \spep[X \sqcup Y] \rightarrow \spep[X] \ot \spep[Y]$ such that
\begin{enumerate}
\item For any two bijections $\sigma: X \rightarrow X'$ and $\tau: Y \rightarrow Y'$ of finite sets,
\begin{align*}
(\spep[\sigma] \ot \spep[\tau]) \circ \Delta_{X,Y}= \Delta_{X',Y'} \circ \spep[\sigma \sqcup \tau]: \spep[X \sqcup Y] \rightarrow \spep[X'] \ot \spep[Y'].
\end{align*}

\item For finite sets $X,Y,Z$,
\begin{align*}
(\Delta_{X,Y} \ot \id_{\spep[Z]}) \circ \Delta_{X \sqcup Y, Z}=(\id_{\spep[X]} \ot \Delta_{Y,Z}) \circ \Delta_{X, Y \sqcup Z}.
\end{align*}

\item There is a linear map $\varepsilon_{\emptyset}: \spep[\emptyset] \rightarrow \bfk$ such that for any finite set $X$,
\begin{align*}
(\varepsilon_{\emptyset} \ot \id_{\spep[X]}) \circ \Delta_{\emptyset, X}=\id_{\spep[X]}=  (\id_{\spep[X]} \ot \varepsilon_{\emptyset}) \circ \Delta_{X, \emptyset}.
\end{align*}
\end{enumerate}

\begin{remark}
\begin{enumerate}
\item The species $\spei$ is a twisted coalgebra, where for $k \in \spei[\emptyset]$,
$$\Delta_{\emptyset, \emptyset}(k)=k(1_{\bf k} \ot 1_{\bf k}) \in \spei[\emptyset] \ot \spei[\emptyset].$$

\item \mcite{Foi} If $\spep$ and $\speq$ are twisted coalgebras, then $\spep \ot \speq$ is also a twisted coalgebra.

\end{enumerate}
\end{remark}

\begin{defn}
	A {\bf bialgebra in the category of species}, or {\bf twisted bialgebra}, is a quintuple $(\spep, m, \ell, \Delta, \varepsilon)$, where $(\spep, m, \ell)$ is a twisted algebra, $(\spep, \Delta, \varepsilon)$ is a twisted coalgebra, and one of the following equivalent conditions holds.
	\begin{enumerate}
		\item $\varepsilon: \spep \rightarrow \spei$ and $\Delta: \spep \rightarrow \spep \ot \spep$ are algebra morphisms.
		
		\item $\ell: \spei \rightarrow \spep$ and $m: \spep \ot \spep \rightarrow \spep$ are coalgebra morphisms.
	\end{enumerate}
\end{defn}

To construct the twisted coalgebra on $\adf$, we define a morphism of species
$$g: O \rightarrow \adf \ot \adf$$
as follows. If $X=\{ x\}$ is a singleton set, then $g_X(\bullet x \bullet)=\bullet x \bullet \ot \bullet + \bullet \ot \bullet x \bullet$. Otherwise, $g_X=0$.
Then by Theorem~\ref{thm:freerb} and Proposition~\ref{prop:ttensor}, there is a unique morphism
$$\Delta: \adf \rightarrow \adf \ot \adf$$
of $\rbs$ such that $g=\Delta \circ i$.

\begin{lemma}
For any finite set $X$ and any $F \in \adf[X]$, 
\begin{align}
(\id \ot \epsilon_{\emptyset}) \circ \Delta_{X,\emptyset}(F)=(\epsilon_{\emptyset} \ot \id) \circ \Delta_{\emptyset, X}(F)=F.
\mlabel{eq:counit}
\end{align}
\mlabel{lem:counit}
\end{lemma}

\begin{proof}
Let us consider the morphism of species $\rho:\adf\longrightarrow \adf$ defined as follows: for any finite set $X$, for any $F\in \adf[X]$,
\[\rho_X(F)=(\epsilon_\emptyset \circ \id)\circ \Delta_{\emptyset,X}(F).\]
By composition, it is an algebra morphism. Let us prove that it is a Rota-Baxter algebra endomorphism of $\adf$. Let $X$ be a finite set and $F\in \adf[X]$. We put $\Delta_{\emptyset,X}(F)=\sum F^{(1)}\otimes F^{(2)}$.
\begin{align*}
\rho_X\circ R_X(F)&=(\epsilon_\emptyset \circ \id)\circ \Delta_{\emptyset,X}\circ R_X(F)\\
&=(\epsilon_\emptyset \circ R_\emptyset \otimes \id)\circ \Delta_{\emptyset,X}(F)+\epsilon_\emptyset(F^{(1)})\epsilon_\emptyset(\bullet)R_X(F^{(2)})\\
&=0+R_X\left(\epsilon_\emptyset(F^{(1)})F^{(2)}\right)\\
&=R_X\circ \rho_X(F),
\end{align*}
as $\epsilon\circ R=0$. So $\rho$ is a morphism of Rota-Baxter species. Moreover, for any singleton set $X=\{x\}$, 
\begin{align*}
\Delta_{\emptyset,X}(\bullet x\bullet)&=\bullet \otimes \bullet x\bullet,
\end{align*}
so $\rho_X(\bullet x\circ \bullet)=x\bullet x\circ \bullet$. In other words, $\rho\circ i=\id_\adf \circ i$. By the unicity part of the universal property, $\rho=\id_\adf$. Consequently, for any finite set $X$ and any $F \in \adf[X]$, 
\[(\id \ot \epsilon_{\emptyset}) \circ \Delta_{X,\emptyset}(F)=F.\]

We now consider the morphism of species $\rho':\adf\longrightarrow \adf$ defined as follows: for any finite set $X$, for any $F\in \adf[X]$,
\[\rho'_X(F)=(\id \circ \epsilon_\emptyset \circ)\circ \Delta_{X,\emptyset}(F).\]
By composition, it is an algebra morphism. Let us prove that it is a Rota-Baxter algebra endomorphism of $\adf$. Let $X$ be a finite set and $F\in \adf[X]$. We put $\Delta_{X,\emptyset}(F)=\sum F_{(1)}\otimes F_{(2)}$.
\begin{align*}
\rho'_X\circ R_X(F)&=(\id \otimes \epsilon_X)\circ \Delta_{X,\emptyset}\circ R_X(F)\\
&=R_X(F_{(1)})\epsilon_\emptyset(F_{(2)})+\epsilon_X(F_{(1)})\bullet \epsilon_\emptyset \circ R_\emptyset(F_{(2)})\\
&=R_X\left(F_{(1)}\epsilon_\emptyset(F_{(2)}\right)+0\\
&=R_X\circ \rho'_X(F),
\end{align*}
as $\epsilon\circ R=0$. So $\rho'$ is a Rota-Baxter species. As $\Delta_{X,\emptyset}(\bullet x\bullet )=\bullet x\bullet\otimes \bullet$, we conclude as for $\rho$ that $\rho'=\id_\adf$, which gives the missing equality.  
\end{proof}

\begin{lemma}
For any finite set $X$ with $X_1 \sqcup X_2 =X$ and any $F \in \adf[X]$, assume $\Delta_{X_1, X_2}(F)=\sum \limits_{(F)} F_{(1)} \ot F_{(2)}$, then
\begin{align}
\Delta_{X_1, X_2}(\epsilon_X(F) \bullet)=\sum\limits_{(F)} \epsilon_{X_1}(F_{(1)}) \bullet \ot \epsilon_{X_2}(F_{(2)}) \bullet.
\label{eq:dvb}
\end{align}
\label{lem:dvb}
\end{lemma}

\begin{proof}
By the proof of Theorem~\mref{thm:freerb}, $\Delta: \adf \rightarrow \adf \ot \adf$ is defined by induction on $|X|$.

If $|X|=0$, then $X=\emptyset$. For a rooted forest $F \in \adf[\emptyset]$, $\Delta_{\emptyset, \emptyset}(F)$ is defined by induction on $\dep(F)$. If $\dep(F)=0$, then $F=\bullet$ and
\begin{align*}
\Delta_{\emptyset, \emptyset}(\bullet)=\bullet \ot \bullet.
\end{align*}
Thus $\Delta_{\emptyset, \emptyset}(\epsilon_{\emptyset} (\bullet) \bullet)=\Delta_{\emptyset, \emptyset} (\bullet) =\bullet \ot \bullet =\epsilon_{\emptyset}(\bullet) \bullet \ot \epsilon_{\emptyset}(\bullet) \bullet$, Eq.~\meqref{eq:dvb} holds.

For the induction step of $\dep(F)>0$, assume $F=B^+(\overline{F})$ for some $\overline{F} \in \adf[\emptyset]$. Assume $\Delta_{\emptyset, \emptyset} (\overline{F})=\sum\limits_{(\overline{F})} \overline{F}_{(1)} \ot \overline{F}_{(2)}$,  then
\begin{align*}
\Delta_{\emptyset, \emptyset}(F)=R^{(2)}_{\emptyset}  (\Delta_{\emptyset, \emptyset}(\overline{F}))=\sum \limits_{(\overline{F})}(R_{\emptyset}(\overline{F}_{(1)}) \ot \overline{F}_{(2)}+ \epsilon_{\emptyset}(\overline{F}_{(1)})\bullet \ot R_{\emptyset}(\overline{F}_{(2)})).
\end{align*}
Thus both sides of Eq.~\meqref{eq:dvb} equal to zero.

If $|X|>0$, then at least one of $X_1, X_2$ is nonempty and so $\epsilon_{X_1}(\overline{F}_{(1)})=0$ or $\epsilon_{X_2}(\overline{F}_{(2)})=0$. Thus both sides of Eq.~\meqref{eq:dvb} equal to zero.
\end{proof}

\begin{prop}
The triple $(\adf, \Delta, \epsilon)$ is a twisted coalgebra.
\label{prop:tcoalgebra}
\end{prop}

\begin{proof}
Define another morphism of species
\begin{align*}
	g': &O \rightarrow \adf \ot \adf \ot \adf,\\ g'_X(F):=&\left\{\begin{array}{ll}\bullet x \bullet \ot \bullet \ot \bullet + \bullet \ot \bullet x \bullet \ot \bullet + \bullet \ot \bullet \ot \bullet x \bullet, & X=\{ x\}, F=\bullet x \bullet,\\
	0, & \text{otherwise.}
	\end{array} \right .
\end{align*}
Then by Theorem~\mref{thm:freerb} and Proposition~\mref{prop:trb2}, there is a unique morphism
$$\Delta': \adf \rightarrow \adf \ot \adf \ot \adf$$
of $\rbs$ such that $g'=\Delta' \circ i$.

As $\Delta: \adf \rightarrow \adf \ot \adf$ is a morphism of $\rbs$, $$(\Delta \ot \id) \circ \Delta: \adf \rightarrow \adf \ot \adf \ot \adf$$
and
$$(\id \ot \Delta) \circ \Delta: \adf \rightarrow \adf \ot \adf \ot \adf$$
are morphisms of twisted algebras. For $F \in \adf[X]$ with $X_1 \sqcup X_2 \sqcup X_3=X$, we have
\begin{align*}
&\ (\Delta_{X_1,X_2} \ot \id) \circ \Delta_{X_1 \sqcup X_2, X_3} \circ R_X(F)\\
=&\ (\Delta_{X_1, X_2} \ot \id) \circ R^{(2)}_X \circ \Delta_{X_1 \sqcup X_2, X_3}(F) \quad \text{(by the proof of Theorem~\mref{thm:freerb})}\\
=&\ (\Delta_{X_1, X_2} \ot \id) \circ (R_{X_1 \sqcup X_2} \ot \id+ \bullet \epsilon_{X_1 \sqcup X_2} \ot R_{X_3}) \circ \Delta_{X_1 \sqcup X_2, X_3}(F)\\
=&\ \left((\Delta_{X_1, X_2} \circ R_{X_1 \sqcup X_2}) \ot \id+ (\Delta_{X_1,X_2} \circ (\bullet \epsilon_{X_1 \sqcup X_2}) \ot R_{X_3})\right) \circ \Delta_{X_1 \sqcup X_2,X_3}(F)\\
=&\ \left((R^{(2)}_{X_1 \sqcup X_2} \circ \Delta_{X_1,X_2}) \ot \id+(\bullet \epsilon_{X_1} \ot \bullet \epsilon_{X_2} \ot R_{X_3})\circ( \Delta_{X_1,X_2} \ot \id) \right) \circ \Delta_{X_1 \sqcup X_2,X_3}(F) \quad \text{(by Lemma~\mref{lem:dvb})}\\
=&\ (R_{X_1} \ot \id \ot \id+ \bullet \epsilon_{X_1} \ot R_{X_2} \ot \id+ \bullet \epsilon_{X_1} \ot \bullet \epsilon_{X_2} \ot R_{X_3}) \circ( \Delta_{X_1,X_2} \ot \id) \circ \Delta_{X_1 \sqcup X_2,X_3}(F)\\
=&\ R^{(3)}_Xp \circ (\Delta_{X_1, X_2} \ot \id) \circ \Delta_{X_1 \sqcup X_2, X_3}(F).
\end{align*}
Hence $(\Delta \ot \id) \circ \Delta: \adf \rightarrow \adf \ot \adf \ot \adf$ is a morphism of $\rbs$. Similarly, $(\id \ot \Delta) \circ \Delta: \adf \rightarrow \adf \ot \adf \ot \adf$ is also a morphism of $\rbs$. Moreover, for any singleton set $X=\{x\}$ with $X_1 \sqcup X_2 \sqcup X_3=X$, there is
\begin{align*}
 (\Delta_{X_1,X_2} \ot \id) \circ \Delta_{X_1 \sqcup X_2, X_3}(\bullet x \bullet)
= (\id \ot \Delta_{X_2, X_3}) \circ \Delta_{X_1, X_2 \sqcup X_3}(\bullet x \bullet),
\end{align*}
which implies that
$$(\Delta \ot \id) \circ \Delta \circ i=g'=(\id \ot \Delta) \circ \Delta \circ i.$$
By the uniqueness of $\Delta'$ such that $\Delta' \circ i=g'$, we conclude $(\Delta \ot \id) \circ \Delta=\Delta'=(\id \ot \Delta) \circ \Delta$.
Then by Lemma~\mref{lem:counit}, $(\adf, \Delta, \epsilon)$ is a twisted coalgebra.
\end{proof}

\begin{theorem}
The quintuple $(\adf, m, \bullet, \Delta, \epsilon)$ is a twisted bialgebra.
\mlabel{prop:tbialgebra}
\end{theorem}

\begin{proof}
By the construction of $\Delta: \adf \rightarrow \adf \ot \adf$, it is a morphism of twisted algebras.
By Lemma~\ref{lem:eam}, $\epsilon: \adf \rightarrow \spei$ is a morphism of twisted algebras.
By Propositions~\ref{prop:talgebra} and~\ref{prop:tcoalgebra}, we get that $(\adf, m, \bullet, \Delta, \varepsilon)$ is a twisted bialgebra.
\end{proof}

\begin{remark}
\begin{enumerate}
	\item 
By~\cite[\S 8.4.1]{AM10}, if a twisted bialgebra $\spep$ is connected in the sense that $\spep[\emptyset]$ is one-dimensional over the base field, then $\spep$ is a twisted Hopf algebra. In our case, the species $\adf$ is not connected since $\adf[\emptyset]$ contains all trees with one leaf. In particular,
$
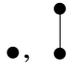
\begin{tikzpicture}[scale = 0.6] %one
	\node at (0,0.1) {$\bullet$};
	\node at (0.3,0) {,};
\end{tikzpicture}
\begin{tikzpicture}[scale = 0.6] %two
%	\node at (0,0) { };
	\node at (0,-0.5) {$\bullet$};
	\node at (0,0.5) {$\bullet$};
	\draw (0,-0.5)-- (0,0.5);
%	\node at (0.3,0) ;
\end{tikzpicture}
$
are both in $\adf[\emptyset]$. Thus one cannot apply that result and conclude that $\adf$ is a twisted Hopf algebra. 
In fact, note that $\adf[\emptyset]=\bfk\{R^n(\bullet),\:n\geq 0\}$ and an easy induction proves that for any $n\geq 0$,
\[\Delta_{\emptyset,\emptyset}(R^n(\bullet))=\sum_{k=0}^n R^k(\bullet)\otimes R^{n-k}(\bullet).\]

\item 
On the other hand, as shown in~\mcite{ZGG,ZXG}, a filtration and coproduct can be given to free Rota-Baxter algebras such that the resulting filtered bialgebra is connected, thereby yielding a Hopf algebra structure on free Rota-Baxter algebras. It would be interesting to lift this construction to species and show that the new Rota-Baxter species is a twisted Hopf algebra. 
\end{enumerate}
\end{remark}

\section{Fock functors and colored Fock functors}
\mlabel{sec:fock}
In this section, we define the Fock functor from Rota-Baxter species to Rota-Baxter graded algebras. As consequence of the bialgebra structure on the free Rota-Baxter species in the previous section, we recover the construction of the free Rota-Baxter algebra~\mcite{EG,Gub,ZXG} by direct construction.

We first recall the concept of Fock functors~\cite{AM10,Foi}. For each $n \geq 0$, denote the set $\{1, \ldots, n \}$ by $\underline{n}$ with the convention $\underline{0}=\emptyset$. For a left $\mathfrak{S}_n$-module $V$, denote the space of coinvariants of $V$ by $\Coinv(V)$, i.e.,
\begin{align*}
\Coinv(V):=V / \mathrm{Vect}(x- \alpha \cdot x, x \in V, \alpha \in \mathfrak{S}_n).
\end{align*}

\begin{defn} \mcite{Foi}
\begin{enumerate}
\item ({\bf Full Fock functor}) Let $\spep$ be a species. Define $\displaystyle \ff (\spep) := \bigoplus_{n=0}^{\infty} \spep[\underline{n}]$. If $f: \spep \rightarrow \speq$ is a morphism of specie, define
\begin{equation*}
\ff (f):
\left\{
\begin{array}{ll}
\ff (\spep)  & \rightarrow \ff (\speq), \\
x \in \spep[ \underline{n}]   & \mapsto f_{\underline{n}} (x) \in \speq[\underline{n}].\\
\end{array}
\right .
\end{equation*}
Then $\ff$ is a functor from the category of species to the category of graded vector spaces.

\item ({\bf Bosonic Fock functor}) Let $\spep$ be a species. Define $\displaystyle \bff(\spep):= \bigoplus_{n=0}^{\infty} \Coinv(\spep[\underline{n}])$. If $f: \spep \rightarrow \speq$ is a morphism of species. Define
\begin{equation*}
\bff (f):
\left\{
\begin{array}{ll}
\bff (\spep)  & \rightarrow \bff (\speq), \\
\overline{x} \in \Coinv(\spep[ \underline{n}])   & \mapsto \overline{f_{\underline{n}} (x)} \in \Coinv(\speq[\underline{n}]).\\
\end{array}
\right .
\end{equation*}
Then $\bff$ is a functor from the category of species to the category of graded vector spaces.

\end{enumerate}
\label{def:functor}
\end{defn}

For $m,n \in \mathbb{N}$, denote by $\sigma_{m,n}:  \underline{m} \sqcup \underline{n} \rightarrow \underline{m+n}$ the bijection
\begin{equation*}
\sigma_{m,n}(i)=
\left\{
\begin{array}{ll}
i  & \text{if $i \in \underline{m}$}, \\
i+m  & \text{if $i \in \underline{n}$}.\\
\end{array}
\right .
\end{equation*}
Then

\begin{lemma}\label{lem:graded} \mcite{Foi}
Let $\spep$ be a twisted algebra (resp. coalgebra, bialgebra).
\begin{enumerate}
\item \label{it:graded}
 $\ff(\spep)$ is a graded bialgebra, in which
\begin{enumerate}
	\item the product given by
    \begin{align*}
    x \cdot y=\spep[\sigma_{m,n}] \circ m_{\spep}(x \ot y), \quad x \in \spep[\underline{m}], y \in \spep[\underline{n}],  m,n \in \mathbb{N};
    \end{align*}
\item the unit is $1_{\spep} \in \spep[\emptyset]$;
\item the coproduct is given by
\begin{align*}
	\Delta(x)=\sum \limits_{i=0}^{m}\spep[\sigma_{i,m-i}^{-1}] \circ \Delta_{\spep}(x), \quad x \in \spep[\underline{m}], m \in \mathbb{N};
\end{align*}
\item  the counit is given by
\begin{align*}
\varepsilon(x)=\varepsilon_{\spep}(x), \quad x \in \ff(\spep).
\end{align*}
\end{enumerate}
\item $\bff(\spep)$ is a graded bialgebra.
\end{enumerate}
\label{lem:aa}
\end{lemma}

\begin{defn}
A {\bf graded Rota-Baxter algebra} is a graded algebra $\displaystyle A=\bigoplus_{n=0}^{\infty} A_n$ with a Rota-Baxter operator $P:A \rightarrow A$ such that $P(A_n) \subseteq A_n$ for each $n \geq 0$.
\end{defn}

\begin{lemma}
Let $(\spep, m, \ell, R)$ be a $\rbs$ of weight $\lambda$. Then
\begin{enumerate}
\item $\ff(\spep)$ is a graded Rota-Baxter algebra of the same weight, with the Rota-Baxter operator $P_{\ff}$ defined by
\begin{align*}
P_{\ff}(x) :=R_{\underline{m}} (x), \quad x \in \spep[\underline{m}], m \in \mathbb{N},
\end{align*}
for the grafting operator $R_{\underline{m}}$ defined in Eq.~\meqref{eq:rbo}.
\mlabel{it:bb1}
\item $\bff(\spep)$ is a graded Rota-Baxter algebra of the same weight, with the Rota-Baxter operator $P_{\bff}$ defined by
    \begin{align*}
    P_{\bff}(x) :=\overline{R_{\underline{m}}(x)}, \quad \overline{x} \in \Coinv(\spep[\underline{m}]), m \in \mathbb{N}.
    \end{align*}
\mlabel{it:bb2}
\end{enumerate}
\label{lem:bb}
\end{lemma}

\begin{proof}
\meqref{it:bb1}. By Lemma~\ref{lem:graded}~\ref{it:graded}, $\ff(\spep)$ is a graded algebra. For $m,n \in \mathbb{N}$ and $x \in \spep[\underline{m}], y \in \spep[\underline{n}]$, we have
\begin{align*}
\ P_{\ff}(x) \cdot P_{\ff}(y)&=\spep[\sigma_{m,n}] \circ m_{\underline{m}, \underline{n}}(R_{\underline{m}}(x) \ot R_{\underline{n}}(y))\\
&=\ \spep[\sigma_{m,n}]\big(R_{\underline{m} \sqcup \underline{n}} (R_{\underline{m}}(x) \cdot y+ x \cdot R_{\underline{n}}(y) + \lambda x \cdot y) \big)\\
&=\ P_{\ff}(P_{\ff}(x) \cdot y+ x \cdot P_{\ff}(y)+ \lambda x \cdot y).
\end{align*}
Hence $\ff(\spep)$ is a graded Rota-Baxter algebra of weight $\lambda$.

\meqref{it:bb2}. Let $I$ be the subspace of $\ff(\spep)$ generated by the elements $x -\spep[\alpha](x)$, where $x \in \spep[\underline{n}]$, $\alpha \in S_n$ with $n \geq 0$. Then $\bff(\spep) =\ff(\spep) /I$. By~\mcite{Foi}, $I$ is an ideal of the algebra $\ff(\spep)$. So we only need to show that $I$ is a Rota-Baxter ideal, i. e., $P_{\ff} (I) \subseteq I$. For any $x \in \spep[\underline{n}]$, $\alpha \in S_n$,
\begin{align*}
P_{\ff}(x- \spep[\alpha] (x))=R_{\underline{n}}(x- \spep[\alpha](x))=R_{\underline{n}}(x)- \spep[\alpha] (R_{\underline{n}}(x)) \in I.
\end{align*}
Then $\bff(\spep)$ is a graded Rota-Baxter algebra.
\end{proof}

Denote by $\mathcal{H}$ the vector space of \pure angularly decorated rooted forests whose angles are decorated by $\{ 1,2, \ldots, n\}$, where $n$ is the number of angles. Denoted by $\widehat{\mathcal{H}}$ the vector space of all angularly decorated forests whose angles are decorated by one element. Then
\begin{prop}
The vector spaces $\mathcal{H}$ and $\widehat{\mathcal{H}}$ are both graded bialgebras and graded Rota-Baxter algebras.
\end{prop}

\begin{proof}
By Theorems~\ref{thm:freerb} and~\ref{prop:tbialgebra}, the species $\adf$ is a twisted Rota-Baxter algebra and a twisted bialgebra. Then by Lemmas~\ref{lem:aa} and~\ref{lem:bb}, $\ff(\adf)$ and $\bff(\adf)$ are graded bialgebras and graded Rota-Baxter algebras. Then by Definition~\ref{def:functor}, $\mathcal{H} \cong \ff(\adf)$ and $\widehat{\mathcal{H}} \cong \bff(\adf)$ as graded vector spaces. Hence $\mathcal{H}$ and $\widehat{\mathcal{H}}$ are both graded bialgebras and graded Rota-Baxter algebras.
\end{proof}

\begin{remark}
It is easy to check the multiplication and Rota-Baxter operator of $\widehat{\mathcal{H}}$ correspond to $\diamond$ and $B^+$, hence $\widehat{\mathcal{H}}$ agrees with the free non-commutative unitary Rota-Baxter algebra generated by one element given in~\mcite{EG,Gub}.
\end{remark}

Let $E$ be a given set. We define a specie $S_E$ in the following way:
\begin{itemize}
\item For any finite set $X$, $S_E[X]$ is the vector space generated by sequences $(e_x)_{x\in X} \in E^X$
of elements of $E$ indexed by $X$. The empty sequence is denoted by $\un\in S_E[\emptyset]$.
\item For any bijection $\sigma:X\longrightarrow Y$ of finite sets and any $(e_x)_{x\in X}\in E^X$, define
\[P[\sigma]((e_x)_{x\in X})=(e_{\sigma^{-1}(y)})_{y\in Y}.\]
\end{itemize}
This species is a twisted bialgebra, where
\begin{itemize}
\item $m_{X,Y}((e_x)_{x\in X}\otimes (f_y)_{y\in Y}))=(g_z)_{z\in X\sqcup Y}$, with
\[g_z=\begin{cases}
e_z\mbox{ if }z\in X,\\
f_z\mbox{ if }z\in Y.
\end{cases}\]
The unit is the empty sequence $\un$.
\item $\Delta_{X,Y}((e_x)_{x\in X\sqcup Y})=(e_x)_{x\in X}\otimes (e_y)_{y\in Y}$.
The counit sends the empty sequence $\un$ to $1$.
\end{itemize}

Next we recall the Hadamard product of species~\cite{AM12, Foi}.
\begin{defn}
Let $\spep$ and $\speq$ be species. The {\bf Hadamard product of $\spep$ and $\speq$} is defined by
\begin{align*}
(\spep \boxtimes \speq)[X]&:= \spep [X]\otimes \speq [X], \, \text{for any finite set $X$}\\
(\spep \boxtimes \speq )[\sigma]&:= \spep [\sigma]\otimes \speq [\sigma], \, \text{for any bijection $\sigma: X \rightarrow Y$}.
\end{align*}
\end{defn}

\begin{remark}
If $\spep$ and $\speq$ are twisted algebras (respectively coalgebras, bialgebras), then $\spep \boxtimes \speq$ is also a twisted algebra
(respectively coalgebra, bialgebra)~\mcite{AM12},  with componently defined product and coproduct.
\end{remark}

If $\spep$ and $\speq$ are two species, for any $n\in \N$,
\begin{align*}
\Coinv(\spep \boxtimes \speq)[\underline{n}]&=
\frac{\spep[\underline{n}]\otimes \speq[\underline{n}]}{\langle \spep[\sigma](p)\otimes \speq[\sigma](q)-p\otimes q\mid
p\in \spep[\underline{n}],\: q\in \speq[\underline{n}],\:\sigma \in \mathfrak{S}_n\rangle}\\
&=\frac{\spep[\underline{n}]\otimes \speq[\underline{n}]}{\langle \spep[\sigma^{-1}](p)\otimes q-p\otimes \speq[\sigma](q)\mid
p\in \spep[\underline{n}],\: q\in \speq[\underline{n}],\:\sigma \in \mathfrak{S}_n\rangle}\\
&=\frac{\spep[\underline{n}]\otimes \speq[\underline{n}]}{\langle p\cdot \sigma\otimes q-p\otimes \sigma\cdot q\mid
p\in \spep[\underline{n}],\: q\in \speq[\underline{n}],\:\sigma \in \mathfrak{S}_n\rangle}\\
&=\spep[\underline{n}]\otimes_{\mathfrak{S}_n} \speq[\underline{n}],
\end{align*}
where the right action of $\mathfrak{S}_n$ on $\spep[\underline{n}]$ is given by
\[p\cdot \sigma=\spep[\sigma^{-1}](p),\]
whereas the left action of $\mathfrak{S}_n$ on $\speq[\underline{n}]$ is given by
\[\sigma\cdot q=\speq[\sigma](q).\]

The $E$-colored Fock functor $\bff_E$ sends a species $\spep$ to
\[\bff_E(\spep)=\bff(\spep \boxtimes S_E)=\bigoplus_{n\in \N} \spep[\underline{n}]\otimes_{\mathfrak{S}_n} S_E [\underline{n}].\]
If $f:\spep \longrightarrow \speq$ is a species morphism,
for any $\overline{p\otimes (e_i)_{i\in \underline{n}}}\in \spep[\underline{n}]\otimes_{\mathfrak{S}_n} S_E [\underline{n}]$,
\[\bff_E(f)(\overline{p\otimes (e_i)_{i\in \underline{n}}})
=\overline{f(p)\otimes (e_i)_{i\in \underline{n}}} \in \bff_E(\speq).\]
If $E$ is a singleton, then $S_E[\underline{n}]$ is a trivial $\mathfrak{S}_n$-module of dimension 1 for any $n\in \N$,
and $\bff_E$ is naturally isomorphic to $\bff$.

As $S_E$ is a twisted bialgebra, we have
\begin{itemize}
\item If $\spep$ is a twisted algebra, then $\spep \boxtimes S_E$ is a twisted algebra, so $\bff_E(\spep)$ is an algebra.
\item If $\spep$ is a twisted coalgebra, then $\spep \boxtimes S_E$ is a twisted coalgebra, so $\bff_E(\spep)$ is a coalgebra.
\item If $\spep$ is a twisted bialgebra, then $\spep \boxtimes S_E$ is a twisted bialgebra, so $\bff_E(\spep)$ is a bialgebra.
\end{itemize}

Note that for $\adf$, $\bff_E(\mathrm{\adf})$ is just the vector space of angularly $E$-decorated forests. Hence, by Proposition~\ref{prop:tbialgebra}, we recover the bialgebra structure of angularly decorated forests.
\begin{coro}
The vector space $\bff_E(\mathrm{ADF})$ of angularly $E$-decorated forests, equipped with the operations defined above, is a bialgebra.
\end{coro}

For $\overline{x \ot (e_i)_{i \in \underline{m}}} \in \adf[\underline{m}] \ot_{\mathfrak{S}_m} S_E[\underline{m}] $, define
\begin{align*}
P(\overline{x \ot (e_i)_{i \in \underline{m}}}) =\overline{R_{\underline{m}}(x) \ot (e_i)_{i \in \underline{m}}}.
\end{align*}
Then the Rota-Baxter algebra structure of angularly decorated forests is recovered:

\begin{prop}
Equipped with the above operator $P$, the vector space $\F_E(\mathrm{ADF})$ of angularly $E$-decorated forests is a Rota-Baxter algebra.
\end{prop}

\begin{proof}
We need to prove that $P$ is well-defined. For any bijection $\sigma: \underline{m} \rightarrow \underline{m}$ such that $\adf[\sigma^{-1}](x)=x'$ and $\adf[\sigma]((e_i)_{i \in \underline{m}})=(e'_i)_{i \in \underline{m}}$. By the definition of $R_{\underline{m}}$, we have $\adf[\sigma^{-1}](R_{\underline{m}}(x))=R_{\underline{m}} (\adf[\sigma^{-1}](x))= R_{\underline{m}}(x')$. Hence
\begin{align*}
P(\overline{x \ot (e_i)_{i \in \underline{m}}}) =\overline{R_{\underline{m}}(x) \ot (e_i)_{i \in \underline{m}}}=\overline{R_{\underline{m}}(x') \ot (e'_i)_{i \in \underline{m}}}=P(\overline{x' \ot (e'_i)_{i \in \underline{m}}}).
\end{align*}
Then by the multiplication of twisted algebras and $\adf$ being a $\rbs$, $\bff_E(\mathrm{ADF})$ is a Rota-Baxter algebra .
\end{proof}

\begin{remark}
Note that the elements of $\bff_E(\mathrm{ADF})$ are just the linear combinations of angularly $E$-decorated forests, and the multiplication and Rota-Baxter operator of $\bff_E(\mathrm{ADF})$ correspond to the product $\diamond$ defined in Eq.~\meqref{eq:pro} and the grafting operator $B^+$. Hence $\bff_E(\mathrm{ADF})$ corresponds the free noncommutative unitary Rota-Baxter algebra generated by the set $E$ as in~\mcite{EG,Gub}.
\end{remark}

\noindent {\bf Acknowledgments. }
This research is supported by NNSFC (12301025 and 12101316)).

\noindent
{\bf Declaration of interests. } The authors have no conflicts of interest to disclose.

\noindent
{\bf Data availability. } Data sharing is not applicable as no new data were created or analyzed.

\end{document}